%% file: ms.tex
\documentclass[a4paper, reqno]{amsart}

\usepackage[UKenglish]{babel}
\usepackage{amsmath, amssymb, amsthm}
\usepackage{scrextend}
\usepackage[hidelinks]{hyperref}
\usepackage{xpatch}
\usepackage{color}
\usepackage{tikz-cd}
\usepackage{enumitem}
\usepackage{theoremref}
\usepackage{mathtools}

\usepackage[inline,final]{showlabels}
\showlabels{thlabel}

\setlist[description]{font=\normalfont\scshape}

\xpatchcmd{\proof}{\itshape}{\normalfont\bfseries}{}{}
\newtheoremstyle{repeat}{}{}{\itshape}{}{\bfseries}{.}{.5em}{#3, repeated}

\newtheorem{theorem}{Theorem}[section]
\newtheorem{proposition}[theorem]{Proposition}
\newtheorem{lemma}[theorem]{Lemma}
\newtheorem{corollary}[theorem]{Corollary}
\newtheorem{fact}[theorem]{Fact}

\theoremstyle{definition}
\newtheorem{definition}[theorem]{Definition}
\newtheorem{remark}[theorem]{Remark}
\newtheorem{convention}[theorem]{Convention}
\newtheorem{example}[theorem]{Example}
\newtheorem{question}[theorem]{Question}

\theoremstyle{repeat}
\newtheorem*{repeated-theorem}{Repeat}

\newcommand{\A}{\mathcal{A}}
\newcommand{\C}{\mathcal{C}}
\newcommand{\M}{\mathcal{M}}
\renewcommand{\L}{\mathcal{L}}

\newcommand{\R}{\mathbb{R}}

\DeclareMathOperator{\tp}{tp}
\DeclareMathOperator{\qftp}{qftp}
\DeclareMathOperator{\gtp}{gtp}
\DeclareMathOperator{\dcl}{dcl}
\DeclareMathOperator{\Hom}{Hom}
\DeclareMathOperator{\colim}{colim}
\DeclarePairedDelimiter{\linspan}{\langle}{\rangle}

\newcommand{\Bil}{\mathbf{Bil}}
\newcommand{\Bils}{\mathbf{Bil}^\textup{s}}
\newcommand{\Bila}{\mathbf{Bil}^\textup{a}}
\renewcommand{\Vec}{\mathbf{Vec}}

\renewcommand{\phi}{\varphi}
\newcommand{\ACF}{{\textup{ACF}}}
\newcommand{\ec}{{\textup{ec}}}

\def\Ind#1#2{#1\setbox0=\hbox{$#1x$}\kern\wd0\hbox to 0pt{\hss$#1\mid$\hss}
\lower.9\ht0\hbox to 0pt{\hss$#1\smile$\hss}\kern\wd0}
\def\ind{\mathop{\mathpalette\Ind{}}}
\def\Notind#1#2{#1\setbox0=\hbox{$#1x$}\kern\wd0\hbox to 0pt{\mathchardef
\nn="3236\hss$#1\nn$\kern1.4\wd0\hss}\hbox to 0pt{\hss$#1\mid$\hss}\lower.9\ht0
\hbox to 0pt{\hss$#1\smile$\hss}\kern\wd0}
\def\nind{\mathop{\mathpalette\Notind{}}}

\title{Bilinear spaces over a fixed field are simple unstable}
\author{Mark Kamsma}
\email{mark@markkamsma.nl}
\urladdr{https://markkamsma.nl}
\address{School of Mathematics, University of East Anglia, Norwich, Norfolk, NR4 7TJ, UK}
\date{\today. \emph{MSC2020}: Primary: 03C45; secondary: 03C10}
\thanks{This project was supported by EPSRC grant EP/W522314/1.}

\begin{document}

% Abstract
\input{tex/abstract.tex}

% Title
\maketitle

% Table of contents
\setcounter{tocdepth}{1} % Only show sections
\tableofcontents

\input{tex/introduction}
\input{tex/preliminaries}
\input{tex/independence-relation}
\input{tex/positive-theory-of-bilinear-spaces}
\input{tex/omega-categoricity}
\input{tex/comparison-to-different-approaches}

% References
\bibliographystyle{alpha}
\bibliography{bibfile}

% Index of all the notes
%\printindex[notes]

\end{document}

%% file: tex/abstract.tex
\begin{abstract}
We study the model theory of vector spaces with a bilinear form over a fixed field. For finite fields this can be, and has been, done in the classical framework of full first-order logic. For infinite fields we need different logical frameworks. First we take a category-theoretic approach, which requires very little set-up. We show that linear independence forms a simple unstable independence relation. With some more work we then show that we can also work in the framework of positive logic, which is much more powerful than the category-theoretic approach and much closer to the classical framework of full first-order logic. We fully characterise the existentially closed models of the arising positive theory. Using the independence relation from before we conclude that the theory is simple unstable, in the sense that dividing has local character but there are many distinct types. We also provide positive version of what is commonly known as the Ryll-Nardzewski theorem for $\omega$-categorical theories in full first-order logic, from which we conclude that bilinear spaces over a countable field are $\omega$-categorical.
\end{abstract}

%% file: tex/introduction.tex
\section{Introduction}
\label{sec:introduction}
Vector spaces with a bilinear form, or \emph{bilinear spaces} as we will call them, appear in many different places in mathematics. Examples include inner product spaces (such as Hilbert spaces) and symplectic spaces arising from symplectic geometry. The model theory of bilinear spaces has been studied using various approaches. One approach is to study $K$-bilinear spaces over some fixed finite field $K$ \cite{kantor_aleph_0-categorical_1989, cherlin_finite_2003}. As $K$ is finite, its elements can simply be named in the signature (e.g.\ as constants). The arising theory turns out to be simple unstable. In \cite{granger_stability_1999} another approach is taken to study bilinear spaces over infinite fields. Their set-up is to consider a bilinear space as a two-sorted structure, with one vector space sort $V$ and one field sort $K$. They prove that the arising theory is non-simple. Later it was shown in \cite{chernikov_model-theoretic_2016} that the theory is NSOP$_1$ when $K$ is algebraically closed.

When studying the model theory of vector spaces we generally take a one-sorted approach, where scalar multiplication is coded by introducing a unary function symbol for each field element. The arising theory is well known to be very well-behaved (i.e.\ stable), regardless of the field. In contrast, when taking a two-sorted approach the arising theory will be at least as complicated as the theory of the field sort. So by making the field part of the language, and thus fixing it, it no longer complicates the arising theory. The main idea of this paper is to do something similar for bilinear spaces, namely fix the field and then study its model theory.

The main problem with the classical first-order approach to vector spaces with a bilinear form $[\cdot, \cdot]$ over some infinite field $K$, such as the approach in \cite{granger_stability_1999}, is the strength of compactness. If $[x, y] = \lambda$ is definable for every $\lambda \in K$, as it should be in any reasonable signature, then the set $\{[x, y] \neq \lambda : \lambda \in K\}$ has a realisation. This means that either the field needs to vary with the models, as in the approach in \cite{granger_stability_1999}, or there are models where the bilinear form $[\cdot, \cdot]$ is incomplete. In this paper we sidestep these issues by working in different logical frameworks. We first take a category-theoretic approach. This is a very general framework that requires very little set-up. Then we consider the more powerful setting of positive logic, which is much closer to the classical framework of full first-order logic, but still allows us to fix the field $K$ to be any field we like.

Independence relations are a central tool in determining where a theory belongs in Shelah's stability hierarchy, at least in the class of NSOP theories. General theory for simple independence relations has been developed for positive logic \cite{pillay_forking_2000, ben-yaacov_simplicity_2003, ben-yaacov_thickness_2003} and for the category-theoretic approach \cite{kamsma_kim-pillay_2020}. These results are roughly of the form ``a given theory/category can only have one nice enough independence relation (the \emph{canonical independence relation}), which reveals its place in the stability hierarchy''. In \cite{haykazyan_existentially_2021} such tools are used to study the positive theory of exponential fields. In this paper we will, in a similar way, employ these tools to show that linear independence is the canonical independence relation in bilinear spaces over a fixed field, and that this implies simplicity and non-stability.

\textbf{Main results.} We study the category $\Bil_K$ of $K$-bilinear spaces with bilinear monomorphisms (injective linear maps that respect the bilinear form). We write $\Bils_K$ and $\Bila_K$ for the full subcategories of symmetric and alternating $K$-bilinear spaces respectively, and we write $\Bil^*_K$ when we mean any of these three categories. To study the model-theoretic behaviour of bilinear spaces using a category-theoretic approach we use the framework of \emph{abstract elementary categories}, or \emph{AECats}, from \cite{kamsma_kim-pillay_2020}, see section \ref{subsec:preliminaries-aecats} for the relevant details.
\begin{theorem}
\thlabel{thm:category-of-bilinear-spaces-is-simple-aecat}
The category $\Bil^*_K$ is an AECat with the amalgamation property that has a canonical simple unstable independence relation $\ind$ given by linear independence.
\end{theorem}
We then move to positive logic and define a theory $T_K$ for $K$-bilinear spaces, again with the notation $T_K^s$ and $T_K^a$ for the symmetric and alternating cases respectively and $T_K^*$ for any of these theories. The signature $\L_K$ will be the standard signature for $K$-vector spaces, i.e.\ where we have a unary function symbol for scalar multiplication for each $\lambda \in K$, a symbol $\neq$ for inequality and the bilinear form is coded by binary relation symbols of the form $[x, y] = \lambda$ for every $\lambda \in K$. The important part is that $[x, y] \neq \lambda$ will not be positively definable when $K$ is infinite, resulting in the fact that the e.c.\ (existentially closed) models are $K$-bilinear spaces. In fact, we fully characterise the e.c.\ models.
\begin{theorem}
\thlabel{thm:ec-models}
The following are equivalent for an $\L_K$-structure $V$:
\begin{enumerate}[label=(\roman*)]
\item $V$ is an e.c.\ model of $T_K^*$,
\item $V$ is a finitely injective $K$-bilinear space,
\item $V$ is an infinite dimensional non-degenerate finitely injective $K$-bilinear space.
\end{enumerate}
If the theory $T_K^*$ in (i) is $T_K^s$ or $T_K^a$ then the conditions in (ii) and (iii) should be further restricted to symmetric or alternating $K$-bilinear spaces respectively.
\end{theorem}
We then prove some results about $T_K^*$ that are useful for technical reasons. We prove that the quantifier-free part of a type determines the entire type (\thref{thm:quantifier-free-type-determines-type}), but $T_K^*$ has positive quantifier elimination precisely when $K$ is finite (\thref{thm:quantifier-elimination}). If $K$ is finite then $T_K^*$ is Boolean (every full first-order formula is equivalent to a positive formula), so we are essentially back in the classical first-order case mentioned earlier in this introduction. However, when $K$ is infinite then $T_K^*$ cannot be Boolean (not even \emph{Hausdorff}, see \thref{def:hausdorff}), but we still have that equality of types is type-definable (i.e.\ \emph{semi-Hausdorff}, see \thref{def:hausdorff}). Putting everything together we can cast \thref{thm:category-of-bilinear-spaces-is-simple-aecat} in the setting of positive logic.
\begin{theorem}
\thlabel{thm:positive-theory-simple-unstable}
The theory $T_K^*$ is simple unstable, and non-dividing coincides with linear independence. That is, for every $\bar{a}$, $\bar{b}$ and $C$ in some e.c.\ model $M$ we have $\bar{a} \ind_C^M \bar{b}$ iff $\tp(\bar{a}/C\bar{b})$ does not divide over $C$. 
\end{theorem}
In the above $\ind$ is the relation given by linear independence, see \thref{def:linear-independence}. Dividing, simplicity, unstability and the type $\tp(\bar{a}/C\bar{b})$ should all be read in the sense of positive logic, see \thref{def:type} and the definitions at the end of section \ref{sec:bilinear-spaces-in-positive-logic}.

We also extend \cite[Theorem 6.5]{haykazyan_spaces_2019} to \thref{thm:omega-categoricity}, resulting in a characterisation of $\omega$-categorical positive theories (i.e.\ theories where there is exactly one countable e.c.\ model, up to isomorphism). This includes a positive variant of being an isolated type. Using this reformulation we then conclude that $T_K^*$ is $\omega$-categorical (for countable $K$), see \thref{cor:bilinear-spaces-omega-categorical}.

\textbf{Overview.} We start with some preliminaries in section \ref{sec:preliminaries} about bilinear spaces, the category-theoretic framework of AECats and positive logic. Then in section \ref{sec:independence-relation} we establish the properties that make linear independence a simple unstable independence relation in bilinear spaces (over a fixed field). In section \ref{sec:bilinear-spaces-in-positive-logic} we study the positive theory of bilinear spaces over a fixed field. In the stand-alone section \ref{sec:omega-categoricity} we prove a theorem characterising $\omega$-categorical theories. Finally, in section \ref{sec:comparison-to-different-approaches} we discuss two other model-theoretic approaches to certain bilinear spaces, and compare them to our approach: Hilbert spaces in continuous logic and the approach of \cite{granger_stability_1999} to bilinear spaces over an infinite field.

\textbf{Acknowledgements.} I would like to thank Jan Dobrowolski, Jonathan Kirby and Rosario Mennuni for their feedback on earlier versions of this paper. I would also like to thank the anonymous referee for their comments which have helped improve the presentation of this paper.

%% file: tex/preliminaries.tex
\section{Preliminaries}
\label{sec:preliminaries}
We discuss some preliminaries about bilinear spaces and the logical frameworks of AECats and positive logic. All claims made in this section are well-known and are either trivial or can be found in the references given at the start of each subsection.

Lowercase letters such as $a, b, c$ and $x, y, z$ will generally denote single elements or variables. We write $\bar{a}$ to mean a (possibly empty, possibly infinite) tuple. The exact length of a tuple often does not matter, so we write $\bar{a} \in A$ instead of $\bar{a} \in A^n$. We write unions in juxtaposition, so $AB$ means $A \cup B$.

\input{tex/preliminaries-bilinear-spaces}
\input{tex/preliminaries-category-theory}
\input{tex/preliminaries-positive-logic}

%% file: tex/preliminaries-bilinear-spaces.tex
\subsection{Bilinear spaces}
\label{subsec:bilinear-spaces}
Throughout this subsection we fix a field $K$. We will drop the $K$ from any names and terms, e.g.\ we write ``vector space'' instead of ``$K$-vector space''.
\begin{definition}
\thlabel{def:bilinear-space}
Let $V$ be a vector space. A \emph{bilinear form} is a map $[\cdot, \cdot]: V \times V \to K$ that is linear in each argument. That is, for all $x,y,z \in V$ and $\lambda \in K$:
\begin{itemize}
\item $[x, y + z] = [x, y] + [x, z]$ and $[x, \lambda y] = \lambda [x, y]$,
\item $[x + y, z] = [x, z] + [y, z]$ and $[\lambda x, y] = \lambda [x, y]$.
\end{itemize}
A \emph{bilinear space} is a vector space equipped with a bilinear form.
\end{definition}
For any $A \subseteq V$ we write $\linspan{A}$ for the linear span of $A$. If $V$ is a bilinear space we naturally view $\linspan{A}$ as a bilinear space by restricting the bilinear form.
\begin{definition}
\thlabel{def:symmetric-alternating}
Let $V$ be a bilinear space, we call $V$ or the bilinear form on $V$:
\begin{itemize}
\item \emph{symmetric} if $[x, y] = [y, x]$ for all $x,y \in V$;
\item \emph{alternating} if $[x, x] = 0$ for all $x \in V$.
\end{itemize}
\end{definition}
Note that in an alternating bilinear space $V$ we always have $[x, y] = -[y, x]$ for all $x,y \in V$, this is easily seen by expanding $[x+y, x+y]$.
\begin{example}
\thlabel{ex:bilinear-spaces}
We give some common examples of bilinear spaces.
\begin{enumerate}[label=(\roman*)]
\item For any field $K$ we can consider the $n$-dimensional vector space $K^n$ with the dot product as bilinear form, which yields a symmetric $K$-bilinear space.
\item A real Hilbert space $H$ is a real vector space with an inner product such that the associated metric makes it into a complete space. In particular, $H$ together with the inner product is a symmetric real bilinear space.
\item The standard symplectic space $\R^{2n}$: let $x_1, \ldots, x_n, y_1, \ldots, y_n$ be the standard basis, then the bilinear form is determined by $[x_i, y_i] = -[y_i, x_i] = 1$ and $0$ for any other combination of basis vectors.
\end{enumerate}
\end{example}
We will use the following straightforward fact implicitly throughout this paper.
\begin{fact}
\thlabel{fact:bilinear-data-of-basis-determines-bilinear-form}
Let $V$ be a vector space with basis $B$. To specify a bilinear form on $V$ it suffices to specify the value for $[a, b]$ for every pair $a,b \in B$ (extend linearly in each argument). In fact, any bilinear form is uniquely determined by the values $[a, b]$ for all $a,b \in B$. It is symmetric precisely when $[a, b] = [b, a]$ for all $a, b \in B$ and it is alternating precisely when $[a, a] = 0$ and $[a, b] = -[b, a]$ for all $a, b \in B$.
\end{fact}
\begin{definition}
\thlabel{def:non-degenerate}
We call a bilinear space $V$ \emph{non-degenerate} if $[x, y] = 0$ for all $y \in V$ implies $x = 0$ and, symmetrically, $[x, y] = 0$ for all $x \in V$ implies $y = 0$.
\end{definition}
All examples in \thref{ex:bilinear-spaces} are non-degenerate. One might normally only be interested in non-degenerate spaces. However, we will not worry about this so much as every degenerate space embeds into a non-degenerate one. This last claim follows from \thref{thm:ec-models} and \thref{fact:complete-to-ec-model}, but we give a sketch of a direct proof here. Start with some bilinear space $V_0$. We extend $V_0$ to $V_1$ by adding a new basis vector $y_x$ for each $x \in V_0$ and extend the bilinear form so that $x$ and $y_x$ have non-zero bilinear product. We repeat this process $\omega$ times and then take the union of the resulting chain of bilinear spaces.
\begin{definition}
\thlabel{def:category-of-vector-spaces}
Let $K$ be a field. We write $\Vec_K$ for the category of $K$-vector spaces with injective linear maps as arrows.
\end{definition}
\begin{definition}
\thlabel{def:bilinear-homomorphism}
A \emph{bilinear homomorphism} is a linear map between bilinear spaces that respects the bilinear forms. That is, a linear map $f: V \to W$ where $V$ and $W$ are bilinear spaces, such that $[x, y] = [f(x), f(y)]$ for all $x, y \in V$. A \emph{bilinear monomorphism} is an injective bilinear homomorphism.
\end{definition}
\begin{definition}
\thlabel{def:category-of-bilinear-spaces}
Let $K$ be a field. We write $\Bil_K$ for the category of $K$-bilinear spaces with bilinear monomorphisms. Furthermore, we write $\Bils_K$ and $\Bila_K$ for the full subcategory of symmetric and alternating $K$-bilinear spaces respectively. We will write $\Bil^*_K$ if something applies to all of $\Bil_K$, $\Bils_K$ and $\Bila_K$.
\end{definition}

%% file: tex/preliminaries-category-theory.tex
\subsection{Abstract Elementary Categories (AECats)}
\label{subsec:preliminaries-aecats}
We first recall the basic notions concerning accessible categories. A great reference for this is \cite{adamek_locally_1994}.
\begin{definition}
\thlabel{def:lambda-presentable}
Let $\C$ be a category and let $\lambda$ be a regular cardinal. An object $X$ in $\C$ is called \emph{$\lambda$-presentable} if whenever $Y = \colim_{i \in I} Y_i$ is a $\lambda$-directed colimit then every arrow $X \to Y$ factors essentially uniquely as $X \to Y_i \to Y$ for some $i \in I$. Equivalently: $\Hom(X, -)$ preserves $\lambda$-directed colimits.
\end{definition}
\begin{definition}
\thlabel{def:accessible-category}
A category $\C$ is called \emph{$\lambda$-accessible} if:
\begin{enumerate}[label=(\roman*)]
\item $\C$ has $\lambda$-directed colimits;
\item there is a set $\A$ of $\lambda$-presentable objects, such that every object in $\C$ can be written as a $\lambda$-directed colimit of objects in $\A$.
\end{enumerate}
In the case where $\lambda = \omega$ we say that $\C$ is finitely accessible. A category is called \emph{accessible} if it is $\lambda$-accessible for some $\lambda$.
\end{definition}
In \cite[Definition 2.5]{kamsma_kim-pillay_2020} an AECat is defined as a pair of categories $(\C, \M)$ satisfying some properties. We will only be interested in the case where $\C = \M$, so when we say ``$\C$ is an AECat'' we will actually mean ``$(\C, \C)$ is an AECat''. This also allows us to simplify the definition.
\begin{definition}
\thlabel{def:aecat}
An accessible category $\C$ is called an \emph{AECat} if it has all directed colimits and all arrows are monomorphisms.
\end{definition}
\begin{definition}
\thlabel{def:ap}
A category $\C$ is said to have the \emph{amalgamation property} if any span of arrows $Y_1 \leftarrow X \to Y_2$ can be completed to a commuting square.
\end{definition}
There are many examples of AECats with the amalgamation property, such as the motivating example: any category of models of some first-order theory with elementary embeddings. In this paper we will just be interested in $\Vec_K$ and $\Bil_K^*$. Note that the arrows are injective maps in both cases, because all arrows in an AECat have to be monomorphisms. The following is straightforward to check.
\begin{fact}
\thlabel{fact:categories-are-aecats}
For any $K$, the categories $\Vec_K$ and $\Bil^*_K$ are finitely accessible. An object $V$ in any of these categories is $\lambda$-presentable precisely when $\dim(V) < \lambda$. In particular, each of these categories is an AECat.
\end{fact}
It is easily seen that $\Vec_K$ has the amalgamation property. Later, we will see that $\Bil^*_K$ also has the amalgamation property (\thref{prop:independent-amalgamation}). This is mainly relevant for the notion of Galois type, which we define below. The amalgamation property ensures that ``having the same Galois type'' is indeed transitive and thus an equivalence relation. Once again, we simplify the definition for our specific case.
\begin{definition}
\thlabel{def:galois-type}
Let $\C$ be one of $\Vec_K$ or $\Bil^*_K$. Let $V$ and $V'$ be objects in $\C$, $\bar{a} \in V$, $\bar{a}' \in V'$ and let $B \subseteq V, V'$ be a shared subset. Then we say that $\bar{a}$ and $\bar{a}'$ have the same \emph{Galois type over $B$}, and we write
\[
\gtp(\bar{a} / B; V) = \gtp(\bar{a}' / B; V'),
\]
if there are arrows $V \xrightarrow{f} W \xleftarrow{g} V'$ in our category $\C$ that agree on $B$ and are such that $f(\bar{a}) = g(\bar{a}')$. If $B = \emptyset$ we drop it from the notation altogether.
\end{definition}
We can make sense of arbitrary elements and subsets of the objects in $\Vec_K$ and $\Bil_K^*$ because any tuple $\bar{a} \in V$ uniquely determines a subspace $\linspan{\bar{a}} \subseteq V$, which is then again an object in our category. Furthermore, as $\bar{a}$ generates $\linspan{\bar{a}}$ arrows with domain $\bar{a}$ are in one-to-one correspondence with arrows with domain $\linspan{\bar{a}}$.

%% file: tex/preliminaries-positive-logic.tex
\subsection{Positive logic}
\label{subsec:positive-logic}
For a more extensive treatment we refer to \cite{ben-yaacov_positive_2003, poizat_positive_2018}.
\begin{definition}
\thlabel{def:positive-syntax}
Fix a signature $\L$. A \emph{positive existential formula} in $\L$ is one that is built using atomic formulas and $\wedge$, $\vee$, $\top$, $\bot$ and $\exists$. An \emph{h-inductive sentence} is a sentence of the form $\forall \bar{x}(\phi(\bar{x}) \to \psi(\bar{x}))$, where $\phi(\bar{x})$ and $\psi(\bar{x})$ are positive existential formulas. A \emph{positive theory} is a set of h-inductive sentences.
\end{definition}
\begin{convention}
\thlabel{conv:positive-is-standard}
Whenever we say ``formula'' or ``theory'' we will mean ``positive existential formula'' and ``positive theory'' respectively, unless explicitly stated otherwise. This also means that every formula and theory we consider will be implicitly assumed to be positive (existential).
\end{convention}
In full first-order logic we consider elementary embeddings because they preserve and reflect truth of all first-order formulas. We do not have negation in positive logic, so there is a difference between preserving and reflecting truth of formulas.
\begin{definition}
\thlabel{def:homomorphism-immersion}
A function $f: M \to N$ between $\L$-structures is called a \emph{homomorphism} if it preserves truth of formulas. That is, if for every $\phi(\bar{x})$ and every $\bar{a} \in M$ we have
\[
M \models \phi(\bar{a}) \implies N \models \phi(f(\bar{a})).
\]
We call $f$ an \emph{immersion} if the converse implication also holds.
\end{definition}
\begin{definition}
\thlabel{def:ec-model}
We call a model $M$ of $T$ an \emph{existentially closed model} or an \emph{e.c.\ model} if the following equivalent conditions hold:
\begin{enumerate}[label=(\roman*)]
\item every homomorphism $f: M \to N$ with $N \models T$ is an immersion;
\item for every $\bar{a} \in M$ and $\phi(\bar{x})$ such that $M \not \models \phi(\bar{a})$ there is $\psi(\bar{x})$ with $T \models \neg \exists \bar{x}(\phi(\bar{x}) \wedge \psi(\bar{x}))$ and $M \models \psi(\bar{a})$.
\end{enumerate}
\end{definition}
Positive model theory generally studies the e.c.\ models of a theory. Every model can be completed to an e.c.\ model anyway, see the fact below.
\begin{fact}
\thlabel{fact:complete-to-ec-model}
Let $M$ be a model of some theory $T$ then there is a homomorphism $f: M \to N$ such that $N$ is an e.c.\ model of $T$.
\end{fact}
\begin{definition}
\thlabel{def:type}
Let $M$ be an e.c.\ model, $B \subseteq M$ and $\bar{a} \in M$. Then the \emph{type of $\bar{a}$ over $B$} is defined as:
\[
\tp(\bar{a}/B) = \{ \phi(\bar{x}) \text{ with parameters in } B : M \models \phi(\bar{a}) \}.
\]
In other words, it is a maximal consistent set of formulas with parameters in $B$. A \emph{partial type} over $B$ is just any consistent set of formulas over $B$.
\end{definition}
\begin{definition}
\thlabel{def:saturation}
We call an e.c.\ model $M$ of some theory $T$ \emph{$\kappa$-saturated} if any partial type with $< \kappa$ variables and $< \kappa$ parameters from $M$ that is finitely satisfiable in $M$ has a realisation in $M$.
\end{definition}
Because e.c.\ models are generally not the same as just models of some theory there can be h-inductive sentences that are true in all e.c.\ models, but fail in some models. Such sentences can be added to the theory without changing the class of e.c.\ models. It will be useful to have some notation for this.
\begin{definition}
\thlabel{def:t-ec}
Let $T$ be a theory. The \emph{Kaiser hull of $T$} is defined as:
\[
T^\ec = \{\chi \text{ an h-inductive sentence} : M \models \chi \text{ for every e.c.\ model $M$ of $T$}\}.
\]
\end{definition}
The following definitions, except for being Boolean, are taken from \cite{ben-yaacov_thickness_2003}, and are very useful for developing (neo)stability theory for positive logic.
\begin{definition}
\thlabel{def:hausdorff}
Let $T$ be a positive theory. We call $T$:
\begin{itemize}
\item \emph{Boolean} if every formula in full first-order logic is equivalent to some positive existential formula modulo $T$, or equivalently: for every positive existential formula $\phi(\bar{x})$ there is a positive existential $\psi(\bar{x})$ that is equivalent to $\neg \phi(\bar{x})$ modulo $T$;
\item \emph{Hausdorff} if for any two distinct types $p(\bar{x})$ and $q(\bar{x})$ there are $\phi(\bar{x}) \not \in p(\bar{x})$ and $\psi(\bar{x}) \not \in q(\bar{x})$ such that $T^\ec \models \forall \bar{x}(\phi(\bar{x}) \vee \psi(\bar{x}))$;
\item \emph{semi-Hausdorff} if equality of types is type-definable, so there is a partial type $\Omega(\bar{x}, \bar{y})$ such that for any $\bar{a}, \bar{b}$ in some e.c.\ model $M$ we have $\tp(\bar{a}) = \tp(\bar{b})$ if and only if $M \models \Omega(\bar{a}, \bar{b})$;
\end{itemize}
\end{definition}
The reason for the name Hausdorff is that this corresponds to the type spaces being Hausdorff, where formulas correspond to closed sets.
\begin{fact}
\thlabel{fact:boolean-implies-hausdorff-implies-semi-hausdorff}
Boolean implies Hausdorff implies semi-Hausdorff.
\end{fact}
The reader that is familiar with the above terminology might be missing one term: \emph{thickness}, which asserts that being an indiscernible sequence is type-definable and which is again weaker than being semi-Hausdorff. However, we will have no use for that notion here, so we leave it at an honourable mention.

Boolean theories are essentially the classical full first-order theories. Through a process called \emph{positive Morleyisation} \cite[section 2.3]{poizat_positive_2018} we can view any theory in full first-order logic as a Boolean positive theory, and we will implicitly do so.

The following fact is useful for proving or disproving that a theory is Hausdorff.
\begin{fact}[{\cite[Theorem 8]{poizat_positive_2018}}]
\thlabel{fact:hausdorff-iff-aph}
The following are equivalent for a theory $T$:
\begin{enumerate}[label=(\roman*)]
\item $T$ is Hausdorff;
\item any model of $T^\ec$ is an amalgamation base, so any span of homomorphisms $M_1 \leftarrow M \to M_2$ between models of $T$ (not necessarily e.c.) with $M \models T^\ec$ can be amalgamated to $M_1 \to N \leftarrow M_2$ where $N \models T$.
\end{enumerate}
\end{fact}
The following is the positive analogue of being complete in full first-order logic.
\begin{definition}
\thlabel{def:jep}
A theory $T$ has the \emph{joint embedding property}, or \emph{JEP}, if for any two e.c.\ models $M_1$ and $M_2$ there is $N \models T$ with homomorphisms $M_1 \to N \leftarrow M_2$.
\end{definition}

%% file: tex/independence-relation.tex
\section{The independence relation}
\label{sec:independence-relation}
Throughout this section we again fix a field $K$ and drop it from any names and terms (except for the names of our categories).

The notion of linear independence can be formulated as a ternary relation on subsets of vector spaces. We give two equivalent formulations.
\begin{definition}
\thlabel{def:linear-independence}
Let $V$ be a vector space and let $A, B, C \subseteq V$. We say that \emph{$A$ is (linearly) independent from $B$ over $C$}, and write $A \ind^V_C B$, if the following equivalent statements hold:
\begin{enumerate}[label=(\roman*)]
\item $\linspan{AC} \cap \linspan{BC} = \linspan{C}$;
\item given a basis $C_0$ for $\linspan{C}$ and $A_0$ and $B_0$ such that $A_0 C_0$ is a basis for $\linspan{AC}$ and $B_0 C_0$ is a basis for $\linspan{BC}$, we have that $A_0 B_0 C_0$ is a linearly independent set.
\end{enumerate}
Notation: any of $A$, $B$ or $C$ can be replaced by a tuple enumerating them. For example, if $\bar{a}$ enumerates $A$ then $\bar{a} \ind_C^V B$ just means $A \ind_C^V B$.
\end{definition}
The relation $\ind$ really is a ternary relation on subobjects of objects in $\Vec_K$ (again using the idea that an arbitrary subset $A$ is essentially the same as $\linspan{A}$). This independence relation is known to have many desirable properties in $\Vec_K$. Most of these properties are immediate from the definition, for the remainder see for example \cite[Fact 2.1.4]{kim_simplicity_2014}.
\begin{fact}
\thlabel{fact:linear-independence-facts}
The independence relation $\ind$ has the following properties in $\Vec_K$.
\begin{description}
\item[\textsc{Invariance}] For any arrow $f: V \to W$ we have $A \ind_C^V B$ iff $f(A) \ind_{f(C)}^W f(B)$.

\item[\textsc{Monotonicity}] If $A \ind_C^V B$ and $B' \subseteq B$ then $A \ind_C^V B'$.

\item[\textsc{Base Monotonicity}] If $A \ind_C^V B$ and $C \subseteq C' \subseteq B$ then $A \ind_{C'}^V B$.

\item[\textsc{Transitivity}] If $A \ind_B^V C$ and $A \ind_C^V D$ with $B \subseteq C \subseteq D$ then $A \ind_B^V D$

\item[\textsc{Symmetry}] If $A \ind_C^V B$ then $B \ind_C^V A$.

\item[\textsc{Existence}] We always have $A \ind_C^V C$.

\item[\textsc{Finite Character}] If $A \ind_C^V B'$ for all finite $B' \subseteq B$ then $A \ind_C^V B$.

\item[\textsc{Local Character}] For any $A, B \subseteq V$ there is $B' \subseteq B$ with $\dim(\linspan{B'}) \leq \dim(\linspan{A})$ such that $A \ind_{B'}^V B$.

\item[\textsc{Extension}] If $\bar{a} \ind_C^V B$ then for any $D \subseteq V$ there is an extension $V \subseteq W$ with some $\bar{a}'$ in $W$ such that $\gtp(\bar{a} / BC; V) = \gtp(\bar{a}' / BC; W)$ and $\bar{a}' \ind_C^W BD$.

\item[\textsc{Stationarity}] If $\gtp(\bar{a}/C; V) = \gtp(\bar{a}'/C; V)$ then $\bar{a} \ind_C^V B$ and $\bar{a}' \ind_C^V B$ implies $\gtp(\bar{a}/BC; V) = \gtp(\bar{a}'/BC; V)$.
\end{description}
\end{fact}
\begin{definition}
\thlabel{def:stable-independence-relation}
An independence relation $\ind$ satisfying the properties in \thref{fact:linear-independence-facts} is called a \emph{stable independence relation}.
\end{definition}
The fact that $\Vec_K$ has a stable independence relation means that it is model-theoretically very well-behaved. The situation in $\Bil^*_K$ turns out to be a little bit more complicated: we lose the \textsc{Stationarity} property.
\begin{proposition}
\thlabel{prop:failure-of-stationarity}
The \textsc{Stationarity} property fails for $\ind$ over every $C$ in $\Bil^*_K$. That is, for any $C$ in $\Bil^*_K$ there is an extension $C \subseteq V$ with $a, a', b \in V$ such that $\gtp(a/C; V) = \gtp(a'/C; V)$, $a \ind_C^V b$ and $a' \ind_C^V b$, while $\gtp(a/Cb; V) \neq \gtp(a'/Cb; V)$.
\end{proposition}
\begin{proof}
Let $C$ be a bilinear space and introduce new linearly independent vectors $a, a', b$ and set $V = \linspan{Caa'b}$. By construction we then have $a \ind_C^V b$ and $a' \ind_C^V b$. We make $V$ into a bilinear space by setting $[x, c] = [c, x] = 0$ for $x \in \{a, a', b\}$ and $c \in C$. We set $[a', b] = [b, a'] = 1$ (in the alternating case we set $[b, a'] = -1$), for the remainder of the pairs in $\{a, a', b\}$ we take their bilinear product to be $0$.

Define $f: \linspan{Ca} \to V$ to be the identity on $C$ and $f(a) = a'$, and extend linearly. Then $f$ is a bilinear monomorphism. So we can amalgamate $V \supseteq \linspan{Ca} \xrightarrow{f} V$ using \thref{prop:independent-amalgamation} to find $V \xrightarrow{g} W \xleftarrow{h} V$ such that $g|_{\linspan{Ca}} = hf$. So $g$ and $h$ agree on $C$ and $g(a) = h(a')$, and thus $\gtp(a/C; V) = \gtp(a'/C; V)$.

Suppose for a contradiction that $\gtp(a/Cb; V) = \gtp(a'/Cb; V)$, then there are $V \xrightarrow{g} W \xleftarrow{h} V$ that agree on $Cb$ and $g(a) = h(a')$. But then $0 = [a, b] = [g(a), g(b)] = [h(a'), h(b)] = [a', b] = 1$.
\end{proof}
The independence relation on $\Bil_K^*$ is still reasonably nice. We just have to replace \textsc{Stationarity} with some weaker property, namely \textsc{3-amalgamation}. Once again, we give a simplified definition for our situation.
\begin{definition}
\thlabel{def:3-amalgamation}
An independence relation $\ind$ has \textsc{3-amalgamation} if the following holds. Suppose that we have a commuting diagram as below, but without the dashed arrows and without $W$. We view all the arrows as inclusions. Suppose furthermore that $A \ind_D^{V_1} B$, $B \ind_D^{V_3} C$ and $C \ind_D^{V_2} A$. Then we can find the dashed arrows and $W$, such that $A \ind_D^W V_3$ and the resulting diagram commutes.
\[
% https://tikzcd.yichuanshen.de/#N4Igdg9gJgpgziAXAbVABwnAlgFyxMJZABgBoBmAXVJADcBDAGwFcYkQAREAX1PU1z5CKMgEZqdJq3YBBHnxAZseAkXIUJDFm0QgAQvP7KhRUaQBMmqTpABhQ4oErhydeJpbpugGoB9UQ5KgqooZsRW2ux+5oFOJigALBYRXiB+5LHGIchJ4R7W7ADqPBIwUADm8ESgAGYAThAAtkhkIDgQSOa8tQ3NiK3tSOTdIPVNLTSDiKIjY31mbR2ICbO9Q5NLAKyr49MbSNsKc+uLSABsO33m+4gXR2uI16fLlwc3AOw0jFhgNlD0cAAFmUHMdlh8vj8-gDgVBQQ8zhCQN9fux-kCQdxKNwgA
\begin{tikzcd}[row sep=tiny]
                                    & V_2 \arrow[rrr, dashed]  &  &                         & W                      \\
A \arrow[rrr] \arrow[ru]            &                          &  & V_1 \arrow[ru, dashed]  &                        \\
                                    & C \arrow[uu] \arrow[rrr] &  &                         & V_3 \arrow[uu, dashed] \\
D \arrow[rrr] \arrow[ru] \arrow[uu] &                          &  & B \arrow[ru] \arrow[uu] &                       
\end{tikzcd}
\]
\end{definition}
\begin{definition}
\thlabel{def:simple-independence-relation}
An independence relation $\ind$ that satisfies the properties in \thref{fact:linear-independence-facts}, except possibly \textsc{Stationarity}, and also satisfies \textsc{3-amalgamation} is called a \emph{simple independence relation}.
\end{definition}
The possible failure of \textsc{Stationarity} is precisely what distinguishes a simple independence relation from a stable one. This is because we get \textsc{3-amalgamation} from \textsc{Stationarity}, modulo the rest of the properties.
\begin{fact}[{\cite[Proposition 6.16]{kamsma_kim-pillay_2020}}]
\thlabel{fact:stationarity-implies-3-amalgamation}
Any stable independence relation also satisfies \textsc{3-amalgamation}. So every stable independence relation is also simple.
\end{fact}
Simple independence relations are always canonical, in the sense that there can only be one on a given AECat, see the fact below. So $\ind$, being stable and thus in particular simple, is the canonical independence relation on $\Vec_K$. Even though the formulation of the properties is slightly different from the formulation in \cite{kamsma_kim-pillay_2020}, it is an easy exercise to see that they are equivalent.
\begin{fact}[{\cite[Theorem 1.1]{kamsma_kim-pillay_2020}}]
\thlabel{fact:canonicity-of-independence}
Let $\C$ be an AECat with the amalgamation property. If $\ind$ and $\ind'$ are simple independence relations on $\C$ then $\ind = \ind'$.
\end{fact}
The remainder of this section is now devoted to proving that $\ind$ is a simple independence relation on $\Bil^*_K$, and is thus also the canonical independence relation.
\begin{proposition}[Independent amalgamation]
\thlabel{prop:independent-amalgamation}
Let $V, W_1, W_2$ be bilinear spaces. Let $W_1 \xleftarrow{f_1} V \xrightarrow{f_2} W_2$ be bilinear monomorphisms. Then there is a bilinear space $U$ and bilinear monomorphisms $W_1 \xrightarrow{g_1} U \xleftarrow{g_2} W_2$, such that $g_1 f_1 = g_2 f_2$ and $g_1(W_1) \ind_{g_1 f_1(V)}^U g_2(W_2)$. Furthermore, if $V, W_1, W_2$ are symmetric/alternating then we can choose $U$ to be symmetric/alternating.
\end{proposition}
\begin{proof}
After renaming elements we may assume $V \subseteq W_1, W_2$ and $W_1 \cap W_2 = V$. So we can take $U = \linspan{W_1 W_2}$ and take $g_1$ and $g_2$ to be the relevant embeddings. So we have $W_1 \ind_V^U W_2$ and we are left to extend the bilinear form to all of $U$. Let $V' \subseteq V$ be a basis for $V$ and for $i \in \{1, 2\}$ let $W'_i \subseteq W_i$ be such that $V'W'_i$ is a basis for $W_i$. Since $W_1 \ind_V^U W_2$ we have that $V'W'_1 W'_2$ is a linearly independent set, and hence a basis for $U$. So we can set $[w_1, w_2] = [w_2, w_1] = 0$ for all $w_1 \in W_1'$ and $w_2 \in W_2'$, and extend linearly in each argument. The final claim about the symmetric/alternating property then immediately follows.
\end{proof}
\begin{proposition}
\thlabel{prop:extension}
The independence relation $\ind$ on $\Bil^*_K$ satisfies \textsc{Extension}.
\end{proposition}
\begin{proof}
Let $\bar{a} \ind_C^V B$. We will prove that there is some extension $V \subseteq W$ and a tuple $\bar{a}'$ in $W$ such that $\bar{a}' \ind_C^W V$ and $\gtp(\bar{a}/BC; V) = \gtp(\bar{a}'/BC; W)$. Consider the span of inclusions $\linspan{BC\bar{a}} \supseteq \linspan{BC} \subseteq V$ and use \thref{prop:independent-amalgamation} to find $\linspan{BC\bar{a}} \xrightarrow{f} W \xleftarrow{g} V$ completing this span to a commuting square, where we may assume $g$ to be an inclusion, such that $f(\linspan{BC\bar{a}}) \ind_{\linspan{BC}}^W V$. We thus have the required extension $V \subseteq W$, and we set $\bar{a}' = f(\bar{a})$. Then $\gtp(\bar{a}'/BC; W) = \gtp(\bar{a}/BC; \linspan{BC\bar{a}}) = \gtp(\bar{a}/BC; V)$. It follows that $\bar{a}' \ind_C^W \linspan{BC}$ and applying \textsc{Monotonicity} to $f(\linspan{BC\bar{a}}) \ind_{\linspan{BC}}^W V$ yields $\bar{a}' \ind_{\linspan{BC}}^W V$. So by \textsc{Transitivity} we get $\bar{a}' \ind_C^W V$, as required.
\end{proof}
We prove a stronger version of \textsc{3-amalgamation}, where the resulting cube satisfies two extra instances of independence.
\begin{theorem}[\textsc{3-amalgamation}]
\thlabel{thm:3-amalgamation}
Suppose that we have a commuting diagram as below of bilinear monomorphisms (which we view as inclusions), but without the dashed arrows and without $W$. Suppose furthermore that $A \ind_D^{V_1} B$, $B \ind_D^{V_3} C$ and $C \ind_D^{V_2} A$. Then there is a bilinear space $W$, together with the dashed bilinear monomorphisms, such that $A \ind_D^W V_3$, $B \ind_D^W V_2$ and $C \ind_D^W V_1$ and the resulting diagram commutes.
\[
% https://tikzcd.yichuanshen.de/#N4Igdg9gJgpgziAXAbVABwnAlgFyxMJZABgBoBmAXVJADcBDAGwFcYkQAREAX1PU1z5CKMgEZqdJq3YBBHnxAZseAkXIUJDFm0QgAQvP7KhRUaQBMmqTpABhQ4oErhydeJpbpugGoB9UQ5KgqooZsRW2ux+5oFOJigALBYRXiB+5LHGIchJ4R7W7ADqPBIwUADm8ESgAGYAThAAtkhkIDgQSOa8tQ3NiK3tSOTdIPVNLTSDiKIjY31mbR2ICbO9Q5NLAKyr49MbSNsKc+uLSABsO33m+4gXR2uI16fLlwc3AOw0jFhgNlD0cAAFmUHMdlh8vj8-gDgVBQQ8zhCQN9fux-kCQdxKNwgA
\begin{tikzcd}[row sep=tiny]
                                    & V_2 \arrow[rrr, dashed]  &  &                         & W                      \\
A \arrow[rrr] \arrow[ru]            &                          &  & V_1 \arrow[ru, dashed]  &                        \\
                                    & C \arrow[uu] \arrow[rrr] &  &                         & V_3 \arrow[uu, dashed] \\
D \arrow[rrr] \arrow[ru] \arrow[uu] &                          &  & B \arrow[ru] \arrow[uu] &                       
\end{tikzcd}
\]
Furthermore, if $V_1, V_2, V_3$ are all symmetric/alternating then we can choose $W$ to be symmetric/alternating.
\end{theorem}
\begin{proof}
First pick a basis $D'$ of $D$ and let $A', B', C'$ be such that $A'D'$, $B'D'$ and $C'D'$ are bases of $AD$, $BD$ and $CD$ respectively. Because $A \ind_D^{V_1} B$ we have that $A'B'D'$ is a linearly independent set, so we extend it to a basis for $V_1$. So let $V'_1$ such that $A'B'D'V'_1$ is a basis for $V_1$. Similarly we find $V'_2$ and $V'_3$. We may assume that $V'_1, V'_2, V'_3$ are such that $D' A'B'C' V'_1 V'_2 V'_3$ is a basis for $W = \linspan{D' A'B'C' V'_1 V'_2 V'_3}$. This induces canonical inclusions $V_i \subseteq W$. So in particular $V_1 \cap V_2 = A$, $V_1 \cap V_3 = B$ and $V_2 \cap V_3 = C$, as subspaces of $W$. From this the independence relations in the conclusions easily follow. For example: $A \cap V_3 = A \cap V_1 \cap V_3 = A \cap B = D$ hence $A \ind_D^W V_3$, where we used $A \ind_D^{V_1} B$ in the final equality.

We are left to define a bilinear form on $W$. For any $e, e' \in D' A'B'C' V'_1 V'_2 V'_3$ with $e, e' \in V_i$ for some $1 \leq i \leq 3$ the choice for $[e, e']$ is forced, and these choices are compatible by commutativity of the original diagram. For the remainder of the pairs $e, e'$ we set $[e, e'] = 0$ and extend linearly in each argument. The final claim about the symmetric/alternating property then follows by \thref{fact:bilinear-data-of-basis-determines-bilinear-form}.
\end{proof}
\begin{repeated-theorem}[\thref{thm:category-of-bilinear-spaces-is-simple-aecat}]
The category $\Bil^*_K$ is an AECat with the amalgamation property that has a canonical simple unstable independence relation $\ind$ given by linear independence.
\end{repeated-theorem}
\begin{proof}
From \thref{fact:categories-are-aecats} we know that $\Bil^*_K$ is an AECat. We get the amalgamation property from \thref{prop:independent-amalgamation}. That $\ind$ then forms a simple independence relation comes down to checking all the required properties. We only need to verify \textsc{Invariance}, \textsc{Extension} and \textsc{3-amalgamation}, as the remaining properties do not depend on the category we are working in, so we get them directly from \thref{fact:linear-independence-facts}. As any bilinear monomorphism is an injective linear map, \textsc{Invariance} is immediate. The two remaining properties are exactly \thref{prop:extension} and \thref{thm:3-amalgamation}. In \thref{prop:failure-of-stationarity} we saw that \textsc{Stationarity} fails for $\ind$ on $\Bil^*_K$, and so $\ind$ is not stable. Finally, canonicity follows from \thref{fact:canonicity-of-independence}.
\end{proof}

%% file: tex/positive-theory-of-bilinear-spaces.tex
\section{Bilinear spaces in positive logic}
\label{sec:bilinear-spaces-in-positive-logic}
We define and study the positive theory of $K$-bilinear spaces.
\begin{definition}
\thlabel{def:signature-of-bilinear-spaces}
Let $K$ be some field. We define the signature $\L_K$ as follows: $\L_K$ includes the standard signature for $K$-vector spaces, with a function symbol for scalar multiplication for each $\lambda \in K$ and a symbol $\neq$ for inequality. Furthermore, for every $\lambda \in K$ we have a binary relation symbol $[x, y] = \lambda$, which will express that the bilinear product of $x$ and $y$ is $\lambda$.
\end{definition}
We naturally view every $K$-bilinear space as an $\L_K$-structure.
\begin{definition}
\thlabel{def:theories-of-bilinear-spaces}
Let $K$ be some field. We define the following $\L_K$-theories:
\begin{itemize}
\item $T_K$ is the common h-inductive theory of all $K$-bilinear spaces,
\item $T_K^s$ is the common h-inductive theory of all symmetric $K$-bilinear spaces,
\item $T_K^a$ is the common h-inductive theory of all alternating $K$-bilinear spaces.
\end{itemize}
We write $T_K^*$ if a statement applies to all of $T_K$, $T_K^s$ and $T_K^a$.
\end{definition}
Note that having a symbol for inequality forces the homomorphisms between models of $T_K^*$ to be injective.

Any positive existential formula $\phi(\bar{x})$ is equivalent to one of the form
\[
\bigvee_{i = 1}^n \exists \bar{y}_i \psi_i(\bar{x}, \bar{y}_i),
\]
where each $\psi_i(\bar{x}, \bar{y}_i)$ is a conjunction of atomic formulas. We will implicitly use this fact in what follows by assuming all formulas are of this form, and we recall some relevant terminology.
\begin{definition}
\thlabel{def:regular-formula}
A \emph{regular formula} (sometimes also called \emph{pp-formula}) is one of the form $\exists \bar{y} \psi(\bar{x}, \bar{y})$ where $\psi(\bar{x}, \bar{y})$ is a conjunction of atomic formulas.
\end{definition}
\begin{lemma}
\thlabel{lem:negation-of-bilinear-value}
Let $\lambda \in K$ and let $\psi(x, y)$ be a regular formula that contains no linear equations such that $T_K^* \models \neg \exists xy(\psi(x, y) \wedge [x, y] = \lambda)$. Then there is some $\lambda' \in K$ such that $T_K^* \models \forall xy(\psi(x, y) \to [x, y] = \lambda')$. Similarly, for a regular formula $\psi(x)$ that contains no linear equations such that $T_K^* \models \neg \exists xy(\psi(x) \wedge [x, x] = \lambda)$ there is some $\lambda' \in K$ such that $T_K^* \models \forall x(\psi(x) \to [x, x] = \lambda')$.
\end{lemma}
\begin{proof}
The proof below is for the case of $T_K$ and a formula $\psi(x, y)$ in two variables, at the end we discuss how to make the same proof work for a single variable and in $T_K^s$ and $T_K^a$. We can write $\psi(x, y)$ as
\[
\exists \bar{z}( \chi(x, y, \bar{z}) \wedge \bigwedge_{i \in I} [v_i, w_i] = \mu_i ),
\]
where $\chi(x, y, \bar{z})$ is a conjunction of linear inequalities (as $\psi(x, y)$ contains no linear equations) and each $v_i$ and $w_i$ is a linear combination of $x$, $y$ and $\bar{z}$.

We now claim that there is some $\lambda'$ such that for any bilinear space $V$ and any $c, d \in V$ such that $V \models \psi(c, d)$ we have $V \models [c, d] = \lambda'$. This is enough, because then $\forall xy(\psi(x, y) \to [x, y] = \lambda')$ is in $T_K$, as required.

To prove the claim we argue by contradiction. Suppose there are distinct $\lambda_1$ and $\lambda_2$ and bilinear spaces $V_1$ and $V_2$ with $c_1,d_1 \in V_1$ and $c_2, d_2 \in V_2$ such that $V_1 \models \psi(c_1, d_1) \wedge [c_1, d_1] = \lambda_1$ and $V_2 \models \psi(c_2, d_2) \wedge [c_2, d_2] = \lambda_2$. We rename the variables appearing in the quantifier-free part of $\psi$ as follows: $x$ becomes $z_1$, $y$ becomes $z_2$ and then we can enumerate $\bar{z}$ as $z_3, \ldots, z_n$ for some $n$. We introduce a variable $u_{ij}$ for all $1 \leq i,j \leq n$, so letting $u_{ij}$ represent $[z_i, z_j]$ we can view $\bigwedge_{i \in I} [v_i, w_i] = \mu_i$ as a system $S$ of linear equations in variables $(u_{ij})_{1 \leq i,j \leq n}$. This determines an affine space
\[
A = \{ (\alpha_{ij})_{1 \leq i,j \leq n} \in K^{n^2} : (\alpha_{ij})_{1 \leq i,j \leq n} \text{ is a solution to } S \}.
\]
By our assumption about the existence of $V_1$ and $V_2$ we know that there must be solutions $(\alpha_{ij})_{1 \leq i,j \leq n}$ and $(\beta_{ij})_{1 \leq i,j \leq n}$ in $A$ with $\alpha_{1,2} = \lambda_1$ and $\beta_{1,2} = \lambda_2$. The projection of an affine space onto one coordinate must be either a point or all of $K$. So since $\lambda_1$ and $\lambda_2$ are distinct, the projection of $A$ on the coordinate indexed by $(1, 2)$ must be all of $K$. We thus find a solution $(\gamma_{ij})_{1 \leq i,j \leq n}$ with $\gamma_{1,2} = \lambda$. Let $V$ be a vector space with basis $e_1, \ldots, e_n$ and make it into a bilinear space by setting $[e_i, e_j] = \gamma_{ij}$ for all $1 \leq i,j \leq n$. We now have $V \models \psi(e_1, e_2)$, where the existential quantifier over $\bar{z}$ is satisfied by $e_3, \ldots, e_n$. The $\chi$ part is then satisfied because it only contains linear inequalities and $e_1, \ldots, e_n$ are linearly independent, while the part $\bigwedge_{i \in I} [v_i, w_i] = \mu_i$ is satisfied by our choice of $[e_i, e_j]$. At the same time we have $V \models [e_1, e_2] = \lambda$, but this contradicts $T_K \models \neg \exists xy(\psi(x,y) \wedge [x, y] = \lambda)$.

The case where $\psi(x)$ has only one variable is a simpler version of the argument above. We replace any occurrence of the $y$ variable by the $x$ variable, and we will just have that $c_1 = d_1$, $c_2 = d_2$ and $e_1 = e_2$.

To make the above argument work for $T_K^s$ we restrict ourselves to symmetric bilinear spaces in the entire argument. To make sure the $V$ we construct is also symmetric, we can just add equations $u_{ij} = u_{ji}$ to our system $S$ of linear equations. Since the $V_1$ and $V_2$ we find are then assumed to be symmetric bilinear spaces, they still yield solutions to $S$. A similar trick works for $T_K^a$.
\end{proof}
\begin{lemma}
\thlabel{lem:ec-is-bilinear}
Every e.c.\ model of $T_K^*$ is a $K$-bilinear space.
\end{lemma}
\begin{proof}
Let $M$ be an e.c.\ model of $T_K^*$. Clearly $T_K^*$ will already specify that $M$ is a $K$-vector space. We need to prove that the binary relations $[x, y] = \lambda$ encode a bilinear form on $M$. For that we only need to check that for every $a, b \in M$ there is some $\lambda \in K$ such that $M \models [a, b] = \lambda$. Then $T_K^*$ will guarantee that this $\lambda$ is unique and that the binary function $M \times M \to K$ encoded by the relation symbols in $\L_K$ is actually a bilinear form. In fact, it is enough to check this for linearly independent $a$ and $b$, and when $a = b$. Then we can take any basis $B$ of $M$ and since $[\cdot, \cdot]$ will then be defined on any $a, b \in B$ it extends to all of $M$, which is forced by $T_K^*$.

If $M \models [a, b] = 0$ then we are done. So suppose that $M \not \models [a, b] = 0$. Then as $M$ is e.c.\ there must be some $\phi(x, y)$ such that $T_K^* \models \neg \exists xy(\phi(x,y) \wedge [x, y] = 0)$ and $M \models \phi(a, b)$. We can write $\phi(x, y)$ as a disjunction of regular formulas. Let $\psi(x, y)$ be a disjunct so that $M \models \psi(a, b)$, then we have $T_K^* \models \neg \exists xy(\psi(x,y) \wedge [x, y] = 0)$. We can write $\psi(x, y)$ as $\exists \bar{z} \chi(x, y, \bar{z})$, where $\chi(\bar{z})$ is a conjunction of atomic formulas. We may assume that $\chi(x, y, \bar{z})$ does not contain any linear equations. This is because we assumed $a$ and $b$ to be linearly independent, or $a = b$ and then we just drop the $y$ variable, and any $z \in \bar{z}$ that is linearly dependent on the remainder of the variables can be eliminated from the quantifier by replacing it by the appropriate linear combination of variables. By \thref{lem:negation-of-bilinear-value} there is then some $\lambda \in K$ such that $T_K^* \models \forall xy(\psi(x, y) \to [x, y] = \lambda)$. We thus have $M \models [a, b] = \lambda$, as required.
\end{proof}
The following is the analogue of being complete for first-order theories, and in fact together with \thref{prop:boolean-or-semi-hausdorff-not-hausdorff} this implies that $T_K^*$ is complete for finite $K$.
\begin{corollary}
\thlabel{cor:jep}
The theory $T_K^*$ has JEP.
\end{corollary}
\begin{proof}
Use that e.c.\ models are $K$-bilinear spaces (\thref{lem:ec-is-bilinear}) to amalgamate (\thref{prop:independent-amalgamation}) over the trivial bilinear space.
\end{proof}
\begin{proposition}
\thlabel{prop:formula-for-linear-independence}
For every $n \geq 1$ there is a formula $\theta_n(x_1, \ldots, x_n)$ such that for every $a_1, \ldots, a_n$ in any e.c.\ model $M$ of $T_K^*$ we have that $M \models \theta_n(a_1, \ldots, a_n)$ iff $a_1, \ldots, a_n$ are linearly independent.
\end{proposition}
\begin{proof}
We will build the formula $\theta_n(x_1, \ldots, x_n)$ by induction on $n$. For $\theta_1(x_1)$ we take $x_1 \neq 0$. Having built $\theta_n(x_1, \ldots, x_n)$ we define $\theta_{n+1}(x_1, \ldots, x_n, x_{n+1})$ as
\[
\theta_n(x_1, \ldots, x_n) \wedge \exists yz \left( \bigwedge_{i = 1}^n ([y, x_i] = 1 \wedge [z, x_i] = 1) \wedge [y, x_{n+1}] = 1 \wedge [z, x_{n+1}] = 0 \right).
\]
Let $M$ be an e.c.\ model of $T_K^*$ and let $a_1, \ldots, a_{n+1} \in M$. Suppose that $M \models \theta_{n+1}(a_1, \ldots, a_{n+1})$. Then by the inductive hypothesis $a_1, \ldots, a_n$ are linearly independent. So it remains to be shown that $a_{n+1}$ is not a linear combination of $a_1, \ldots, a_n$. Suppose for a contradiction that $a_{n+1} = \lambda_1 a_1 + \ldots + \lambda_n a_n$. Let $b$ and $c$ be realisations for the $y$ and $z$ variables respectively. Then as $[b, a_i] = 1$ for all $1 \leq i \leq n$ we must have that $1 = [b, a_{n+1}] = \lambda_1 + \ldots + \lambda_n$. By similar reasoning we also get $0 = [c, a_{n+1}] = \lambda_1 + \ldots + \lambda_n$, and thus $1 = 0$. So we arrive at a contradiction and conclude that $a_1, \ldots, a_{n+1}$ must indeed be linearly independent.

Conversely, suppose that $a_1, \ldots, a_{n+1}$ are linearly independent. Then by the induction hypothesis $M \models \theta_n(a_1, \ldots, a_n)$. Extend $a_1, \ldots, a_{n+1}$ to a basis $B$ of $M$ and add two new independent vectors $b$ and $c$. Let $V = \linspan{B b c}$, so it extends $M$. We extend the bilinear form on $M$ to $V$ as follows. We set $[b, a_i] = [c, a_i] = 1$ for all $1 \leq i \leq n$, and we set $[b, a_{n+1}] = 1$ and $[c, a_{n+1}] = 0$. We set the values for $[a_i, b]$ and $[a_i, c]$ for $1 \leq i \leq n+1$ according to whether we work in $T_K^s$ or $T_K^a$ (in the case of just $T_K$ it does not matter). Then $V \models \theta_{n+1}(a_1, \ldots, a_{n+1})$ because $b$ and $c$ are realisations for the $y$ and $z$ variables. So by existential closedness $M \models \theta_{n+1}(a_1, \ldots, a_{n+1})$.
\end{proof}
We can now give a full characterisation of the e.c.\ models of $T_K^*$, which will use the following definition (borrowing some category-theoretic terminology).
\begin{definition}
\thlabel{def:finitely-injective}
We call a $K$-bilinear space $V$ \emph{finitely injective} if for any finite dimensional $K$-bilinear spaces $A \subseteq B$ and any bilinear monomorphism $f: A \to V$ we can find a bilinear monomorphism $g: B \to V$ such that $g$ extends $f$. In the symmetric/alternating case we require all the spaces involved to be symmetric/alternating.
\end{definition}
\begin{repeated-theorem}[\thref{thm:ec-models}]
The following are equivalent for an $\L_K$-structure $V$:
\begin{enumerate}[label=(\roman*)]
\item $V$ is an e.c.\ model of $T_K^*$,
\item $V$ is a finitely injective $K$-bilinear space,
\item $V$ is an infinite dimensional non-degenerate finitely injective $K$-bilinear space.
\end{enumerate}
If the theory $T_K^*$ in (i) is $T_K^s$ or $T_K^a$ then the conditions in (ii) and (iii) should be further restricted to symmetric or alternating $K$-bilinear spaces respectively.
\end{repeated-theorem}
\begin{proof}
\underline{(ii) $\Leftrightarrow$ (iii).} The direction (iii) $\Rightarrow$ (ii) is trivial, so we prove the other direction. Let $V$ be a finitely injective $K$-bilinear space. Then $V$ is infinite dimensional, as we can embed spaces of arbitrarily large (finite) dimension in it. For any $a \in V$ let $A = \linspan{a}$ and define $B = \linspan{ab}$ where $b$ is some new vector linearly independent from $a$. Make $B$ into a $K$-bilinear space by setting $[a, b] = 1$ and also $[b, a] = \pm 1$ (the exact value depends on whether we are in the symmetric or alternating case). We can then extend the embedding $A \subseteq V$ to some bilinear monomorphism $g: B \to V$, and so we have $[a, g(b)] \neq 0$ and $[g(b), a] \neq 0$. Since $a$ was arbitrary we conclude that $V$ is non-degenerate.

\underline{(i) $\implies$ (ii).} By \thref{lem:ec-is-bilinear} we already know that every e.c.\ model is in fact a $K$-bilinear space. Let $A \subseteq B$ be finite dimensional $K$-bilinear spaces and let $f: A \to V$ be a bilinear monomorphism. Let $\bar{a}$ be a basis for $A$ and extend it to a basis $\bar{a}\bar{b}$ for $B$. Let $\bar{x}$ and $\bar{y}$ be variables matching $\bar{a}$ and $\bar{b}$ respectively. Let $\chi(\bar{x}, \bar{y})$ be a conjunction of binary relation symbols $[\cdot, \cdot] = \lambda$ capturing all the bilinear products in $\bar{a}\bar{b}$. Amalgamate $V \xleftarrow{f} A \subseteq B$ to $V \xrightarrow{h} W \supseteq B$, using \thref{prop:independent-amalgamation}. We may assume that $W$ is an e.c.\ model (otherwise complete it to one). Consider the formula $\phi(\bar{x}, \bar{y})$ given by $\theta_n(\bar{x}, \bar{y}) \wedge \chi(\bar{x}, \bar{y})$, where $\theta_n$ is the formula from \thref{prop:formula-for-linear-independence} capturing linear independence. Then $W \models \phi(\bar{a}, \bar{b})$ and thus $W \models \exists \bar{y} \phi(\bar{a}, \bar{y})$. As $\bar{a} = hf(\bar{a})$ and $h: V \to W$ is an immersion we have $V \models \exists \bar{y} \phi(f(\bar{a}), \bar{y})$. We thus find $\bar{c} \in V$ with $V \models \phi(f(\bar{a}), \bar{c})$. We can now extend $f$ to $g: B \to V$ by setting $g(\bar{b}) = \bar{c}$ and extend linearly.

\underline{(ii) $\implies$ (i).} Let $V \subseteq W$ be an extension and $W \models \phi(\bar{a})$ for some $\bar{a} \in V$. We may assume that $W$ is an e.c.\ model (otherwise complete it to one), so in particular it is a $K$-bilinear space. Write $\phi(\bar{x})$ as $\exists \bar{y} \psi(\bar{x}, \bar{y})$, where $\psi(\bar{x}, \bar{y})$ is quantifier-free. Let $\bar{b} \in W$ be such that $W \models \psi(\bar{a}, \bar{b})$. Define $A = \linspan{\bar{a}}$ and $B = \linspan{\bar{a}\bar{b}}$. So $A \subseteq V$ and $A \subseteq B$ are finite dimensional $K$-bilinear spaces. As $V$ is finitely injective we find $g: B \to V$ extending the inclusion $A \subseteq V$. As $\psi$ is quantifier-free we have $B \models \psi(\bar{a}, \bar{b})$ and thus $V \models \psi(\bar{a}, g(\bar{b}))$. So $V \models \phi(\bar{a})$, as required.
\end{proof}
\begin{theorem}
\thlabel{thm:quantifier-free-type-determines-type}
In $T_K^*$ all types are determined by their quantifier-free part.
\end{theorem}
\begin{proof}
Let $\bar{a}$ be some tuple in some e.c.\ model $M$ and let $\bar{b}$ be a tuple in some e.c.\ model $N$, such that $\qftp(\bar{a}) = \qftp(\bar{b})$. Define a function $f: \linspan{\bar{a}} \to N$ by $f(\bar{a}) = \bar{b}$ and extend linearly. Then $f$ is a bilinear monomorphism because $\qftp(\bar{a}) = \qftp(\bar{b})$. So we can amalgamate, using \thref{prop:independent-amalgamation}, to find bilinear monomorphisms $M \xrightarrow{g} V \xleftarrow{h} N$ such that $g(\bar{a}) = h(\bar{b})$. As $M$ and $N$ are e.c.\ models, $g$ and $h$ are immersions, so $\tp(\bar{a}) = \tp(\bar{b})$ follows.
\end{proof}
\begin{corollary}
\thlabel{cor:all-types-isolated}
Modulo $(T_K^*)^\ec$ every type in finitely many variables without parameters is equivalent to a formula. If $K$ is finite, this formula can be taken to be quantifier-free.
\end{corollary}
\begin{proof}
Let $\bar{a}$ be a finite tuple in an e.c.\ model $M$. By \thref{thm:quantifier-free-type-determines-type} it is enough to construct a formula $\phi(\bar{x})$ that is equivalent to $\qftp(\bar{a})$ (modulo $(T_K^*)^\ec$). Let $\bar{a}'$ be a maximal linearly independent subtuple of $\bar{a}$. Let $\psi(\bar{x})$ be a conjunction of linear equations capturing how the remainder of $\bar{a}$ depends on $\bar{a}'$. Let $\chi(\bar{x})$ be a conjunction of binary relation symbols $[\cdot, \cdot] = \lambda$ capturing the bilinear products in $\bar{a}'$. Then the formula $\theta_n(\bar{x}') \wedge \psi(\bar{x}) \wedge \chi(\bar{x}')$ is the $\phi(\bar{x})$ we are looking for, where $\theta_n$ is the formula from \thref{prop:formula-for-linear-independence} capturing linear independence. The final claim follows from the fact that linear independence over a finite field can be expressed by a quantifier-free formula, so we can replace $\theta_n$ by this quantifier-free formula.
\end{proof}
\begin{remark}
\thlabel{rem:quantifier-free-type-vs-quantifier-elimination}
In full first-order logic we have that if all types are determined by their quantifier-free part, then the theory has quantifier elimination, see for example \cite[Theorem 8.4.1]{hodges_model_1993}. In positive logic this is no longer true. As we will see below, $T_K^*$ for infinite $K$ is an example. By \thref{thm:quantifier-free-type-determines-type} we do have that in $T_K^*$ every type is determined by its quantifier-free part. However, \thref{thm:quantifier-elimination} shows that we do not have quantifier elimination.
\end{remark}
\begin{definition}
\thlabel{def:positive-quantifier-elimination}
We say that a theory $T$ has \emph{positive quantifier elimination} if for every formula $\phi(\bar{x})$ there is some quantifier-free formula $\psi(\bar{x})$ that is equivalent to $\phi(\bar{x})$ modulo $T^\ec$.
\end{definition}
\begin{theorem}
\thlabel{thm:quantifier-elimination}
The theory $T_K^*$ has positive quantifier elimination iff $K$ is finite.
\end{theorem}
\begin{proof}
We first prove the right to left direction. The quantifier-free type of any $n$-tuple is fully determined by any linear dependencies and the bilinear products in that tuple. As $K$ is finite, there are only finitely many possibilities for this for a fixed $n$. So there are only finitely many quantifier-free $n$-types, and hence finitely many $n$-types by \thref{thm:quantifier-free-type-determines-type}.

Let now $\phi(\bar{x})$ be some formula and write $[\phi(\bar{x})]$ for the set of all types that contain $\phi(\bar{x})$. By the above discussion $[\phi(\bar{x})]$ is finite and, by \thref{cor:all-types-isolated}, for each $p \in [\phi(\bar{x})]$ there is some quantifier-free $\chi_p(\bar{x})$ that is equivalent to $p$. It then immediately follows that $\phi(\bar{x})$ is equivalent to $\bigvee_{p \in [\phi(\bar{x})]} \chi_p(\bar{x})$, modulo $(T_K^*)^\ec$.

We now prove the contrapositive of the converse. So let $K$ be infinite, we prove that for any $n \geq 2$ the formula $\theta_n$ from \thref{prop:formula-for-linear-independence} is not equivalent to a quantifier-free formula, modulo $(T_K^*)^\ec$. We give a proof for $T_K$ and at the end of the proof we describe how to make the proof work for $T_K^s$ and $T_K^a$. Let $n \geq 2$ and suppose that $\theta_n(x_1, \ldots, x_n)$ is equivalent (modulo $T_K^\ec$) to some quantifier free formula $\psi(x_1, \ldots, x_n)$. We can write $\psi$ as a disjunction $\bigvee_{\ell = 1}^k \psi_\ell(x_1, \ldots, x_n)$, where each $\psi_\ell$ is a conjunction of atomic formulas.

Let $1 \leq \ell \leq k$. For each $1 \leq i,j \leq n$ introduce a variable $u_{ij}$ that will represent $[x_i, x_j]$. The atomic formulas in $\psi_\ell$ of the form $[t, s] = \lambda$, where $\lambda \in K$ and $t$ and $s$ are linear combinations of $x_1, \ldots, x_n$, will then determine a linear system of equations $S_\ell$ in variables $(u_{ij})_{1 \leq i,j \leq n}$. We define the affine space
\[
A_\ell = \{ (\alpha_{ij})_{1 \leq i,j \leq n} \in K^{n^2} : (\alpha_{ij})_{1 \leq i,j \leq n} \text{ is a solution to } S_\ell \}.
\]
We distinguish two cases.
\begin{enumerate}[label=(\arabic*)]
\item There is some $\ell$ such that $A_\ell$ is the entire space $K^{n^2}$. This means that $\psi_\ell$ is (equivalent to) some formula purely in the language of vector spaces. Then as linear independence is not definable in vector spaces over an infinite field, there is some vector space $V$ with $v_1, \ldots, v_n \in V$ such that $V \models \psi_\ell(v_1, \ldots, v_n)$ while $v_1, \ldots, v_n$ are not linearly independent. We make $V$ into a bilinear space (the choice of the bilinear form does not matter) and then extend it to an e.c.\ model $M$. This process does not invalidate the truth of $\psi_\ell(v_1, \ldots, v_n)$. So $M \models \psi_\ell(v_1, \ldots, v_n)$ and hence $M \models \psi(v_1, \ldots, v_n)$. However, we do not have $M \models \theta_n(v_1, \ldots, v_n)$, so $\theta_n$ cannot be equivalent to $\psi$ in all e.c.\ models.
\item Every $A_\ell$ is a proper subspace of $K^{n^2}$. As $K$ is infinite, $A = \bigcup_{1 \leq \ell \leq k} A_\ell$ is still a proper subset of $K^{n^2}$. Pick some $(\alpha_{ij})_{1 \leq i,j \leq n} \in K^{n^2} - A$. Let $V = \linspan{a_1, \ldots, a_n}$ be an $n$-dimensional vector space with basis $a_1, \ldots, a_n$. Make $V$ into a bilinear space by setting $[a_i, a_j] = \alpha_{ij}$ for all $1 \leq i,j \leq n$. Extend $V$ to some e.c.\ model $M \supseteq V$. We now have $M \models \theta_n(a_1, \ldots, a_n)$, as $a_1, \ldots, a_n$ are linearly independent. However, for every $1 \leq \ell \leq k$ we have that $([a_i, a_j]_{ij})_{1 \leq i,j \leq n} = (\alpha_{ij})_{1 \leq i,j \leq n}$ is not a solution to $S_\ell$, so $M \not \models \psi_\ell(a_1, \ldots, a_n)$. Hence $M \not \models \psi(x_1, \ldots, x_n)$, which contradicts that $\theta_n(x_1, \ldots, x_n)$ and $\psi(x_1, \ldots, x_n)$ are equivalent in all e.c.\ models.
\end{enumerate}
A similar proof works for $T_K^s$, we just restrict the variables $(u_{ij})_{1 \leq i,j \leq n}$ and the corresponding systems $S_\ell$ and affine subspaces $A_\ell$ to those $i$ and $j$ with $i \leq j$. This suffices as they will encode all the necessary information for a symmetric bilinear form. Similarly, for $T_K^a$ we restrict things to those $i$ and $j$ where $i < j$.
\end{proof}
\begin{proposition}
\thlabel{prop:boolean-or-semi-hausdorff-not-hausdorff}
If $K$ is finite then $T_K^*$ is Boolean. If $K$ is infinite then $T_K^*$ is semi-Hausdorff, but not Hausdorff.
\end{proposition}
\begin{proof}
Suppose that $K$ is finite. Then $[x, y] \neq \lambda$ is equivalent to $\bigvee_{\lambda' \neq \lambda} [x, y] = \lambda'$. As we also have a symbol for inequality, we have that for every atomic formula $\chi(\bar{x})$ there is a positive (quantifier-free) formula that is equivalent to $\neg \chi(\bar{x})$. The same is then true for quantifier-free formulas. We conclude that $T_K^*$ is Boolean by positive quantifier elimination (\thref{thm:quantifier-elimination}).

We now move on to the case where $K$ is infinite. We first prove that $T_K^*$ is semi-Hausdorff. By \thref{thm:quantifier-free-type-determines-type} we only need to prove that having the same quantifier-free type is type-definable. This is clearly true when we restrict to just the language of $K$-vector spaces. So it suffices to show that $[x,y] = [x', y']$ is definable. Consider the formula $\phi(x,y,x',y')$ given by:
\[
\exists zz'([x, y-z] = 0 \wedge [x-z', z] = 0 \wedge [z', z-y'] = 0 \wedge [z'-x', y'] = 0).
\]
Let $M$ be an e.c.\ model and let $a,b,a',b' \in M$ with $M \models \phi(a,b,a',b')$. Let $c,c' \in M$ be such that $[a,b-c] = [a-c',c] = [c', c-b'] = [c'-a',b'] = 0$. We thus get $[a,b] - [a,c] = [a, c] - [c', c] = [c',c] - [c',b'] = [c', b'] - [a',b'] = 0$, and so $[a,b] = [a,c] = [c', c] = [c',b'] = [a',b']$. So $\phi(x,y,x',y')$ does indeed imply $[x,y] = [x', y']$.

Conversely, let $a,b,a',b' \in M$ such that $[a, b] = [a', b'] = \lambda$. Define $A = \linspan{aba'b'}$ and introduce two new linearly independent vectors $c, c'$ and form $B = \linspan{aba'b'cc'}$. Make $B$ into a $K$-bilinear space by setting $[a, c] = [c', c] = [c', b'] = \lambda$. Pick anything for the remainder of the bilinear products (respecting the form having to be symmetric/alternating). As $M$ is e.c.\  it is finitely injective (\thref{thm:ec-models}), so the inclusion $A \subseteq M$ extends to a bilinear monomorphism $g: B \to M$. Then $g(c)$ and $g(c')$ are realisations for $z$ and $z'$ respectively in $\phi$, so $M  \models \phi(a, b, a', b')$.

Now we prove that $T_K^*$ is not Hausdorff. Let $M$ be an e.c.\ model of $T_K^*$. Using the assumption that $K$ is infinite, we can use compactness for full first-order logic to find an elementary extension $N$ of $M$ with linearly independent $a, b \in N$ such that $N \not \models [a, b] = \lambda$ for all $\lambda \in K$. We claim that there are extensions $N_1$ and $N_2$ of $N$, both models of $T_K^*$, such that $N_1 \models [a, b] = 0$ and $N_2 \models [a, b] = 1$. This shows that $T_K^*$ is not Hausdorff by \thref{fact:hausdorff-iff-aph}, as $N_1 \supseteq N \subseteq N_2$ cannot be amalgamated.

Suppose for a contradiction that one of these extensions, say $N_1$, does not exist. Then by compactness there is a finite conjunction $\phi(a, b, \bar{c})$ of atomic formulas that are true in $N$ such that $T_K^* \models \neg \exists xy(\exists \bar{z}\phi(x, y, \bar{z}) \wedge [x, y] = 0)$. We may assume $\phi$ contains no linear equations since $a$ and $b$ are linearly independent and we can replace any $c \in \bar{c}$ that is linearly dependent on $a$, $b$ and the remainder of $\bar{c}$ by replacing it by the appropriate linear combination. We can thus apply \thref{lem:negation-of-bilinear-value} to obtain some $\lambda \in K$ such that $T_K^* \models \forall xy(\exists \bar{z}\phi(x, y, \bar{z}) \to [x, y] = \lambda)$. However, this would imply that $N \models [a, b] = \lambda$, a contradiction.
\end{proof}
We close out this section by recalling the definitions of simplicity (in the sense of \cite{ben-yaacov_simplicity_2003}) and stability for positive theories, and prove that $T_K^*$ is simple unstable.\footnote{Having JEP allows us to work in a monster model, but we have so far not used monster models and it seems unnecessary to introduce them just for the final result.}
\begin{definition}
\thlabel{def:dividing}
Let $M$ be some e.c.\ model and let $\bar{a}, \bar{b} \in M$, $C \subseteq M$. We say that a type $p(x, \bar{b}) = \tp(\bar{a}/C\bar{b})$ \emph{divides over $C$} if there is an extension $M \subseteq N$ with a $C$-indiscernible sequence $(\bar{b}_i)_{i < \omega}$ in $N$ such that $\bigcup_{i < \omega} p(x, \bar{b}_i)$ is inconsistent.
\end{definition}
\begin{definition}
\thlabel{def:simplicity}
A theory $T$ is called \emph{simple} if dividing has \emph{local character}. That is, there is some cardinal $\mu$ such that for any finite $\bar{a}$ in any e.c.\ model $M$ and any $B \subseteq M$ there is $B_0 \subseteq B$ with $|B_0| \leq \mu$ and $\tp(\bar{a}/B)$ does not divide over $B_0$.
\end{definition}
\begin{definition}
\thlabel{def:stability}
A theory $T$ is called \emph{stable} if there is some cardinal $\mu$ such that for all $A \subseteq M$, where $M$ is an e.c.\ model and $|A| \leq \mu$, there are at most $\mu$ different types over $A$ (possibly realised in extensions of $M$).
\end{definition}
The following fact is one half of a Kim-Pillay style theorem for positive logic, which allows us to characterise simple positive theories based on the existence of a simple independence relation. We simplified the statement for our setting and only mentioned one half as that is what we need, the original theorem is much stronger.
\begin{fact}[{\cite[Theorem 1.51]{ben-yaacov_simplicity_2003}}]
\thlabel{fact:kim-pillay-positive-logic}
Let $T$ be a positive theory. If there is a simple independence relation $\ind$ on subsets of the e.c.\ models of $T$ then $T$ is simple and $\ind$ coincides with non-dividing.
\end{fact}
\begin{repeated-theorem}[\thref{thm:positive-theory-simple-unstable}]
The theory $T_K^*$ is simple unstable, and non-dividing coincides with linear independence. That is, for every $\bar{a}$, $\bar{b}$ and $C$ in some e.c.\ model $M$ we have $\bar{a} \ind_C^M \bar{b}$ iff $\tp(\bar{a}/C\bar{b})$ does not divide over $C$. 
\end{repeated-theorem}
\begin{proof}
We get simplicity and the claim about $\ind$ coinciding with non-dividing directly from applying \thref{fact:kim-pillay-positive-logic} to $\ind$, as all the necessary properties have been verified in section \ref{sec:independence-relation}. To translate between the framework of AECats and the framework of positive logic, we note that Galois types coincide with positive types. That is, if $B$ is a common subset of e.c.\ models $M_1$ and $M_2$ and $\bar{a}_1 \in M_1$ and $\bar{a}_2 \in M_2$ then we have $\tp(\bar{a}_1/B) = \tp(\bar{a}_2/B)$ if and only if $\gtp(\bar{a}_1/B; M_1) = \gtp(\bar{a}_2/B; M_2)$ in $\Bil_K^*$. The left to right direction is a straightforward exercise using compactness and the method of diagrams. For the right to left direction we let $M_1 \xrightarrow{f_1} N \xleftarrow{f_2} M_2$ witness the equality of Galois types. Then indeed $\tp(\bar{a}_1/B) = \tp(f_1(\bar{a}_1)/B) = \tp(f_2(\bar{a}_2)/B) = \tp(\bar{a}_2/B)$, because the bilinear monomorphisms $f_1$ and $f_2$ are $\L_K$-homomorphisms and thus immersions since $M_1$ and $M_2$ are e.c.\ models.

We could prove non-stability using the fact that \textsc{Stationarity} fails, as in \thref{prop:failure-of-stationarity}, because in stable positive theories we must have \textsc{Stationarity} over certain sets \cite[Theorem 2.8]{ben-yaacov_simplicity_2003}. However, we will give a direct proof for \thref{def:stability}. Let $\mu$ be any infinite cardinal, and let $M$ be an e.c.\ model with at least $\mu$ many linearly independent vectors $(a_i)_{i < \mu}$. Then for any $\chi: \mu \to \{0,1\}$ we define the partial type $\Sigma_\chi(x) = \{[x, a_i] = \chi(i) : i < \mu\}$. Then each $\Sigma_\chi(x)$ can be extended to a type $p_\chi(x)$ over $(a_i)_{i < \mu}$ (i.e.\ it will have a realisation in some extension of $M$). This yields an injection from $2^\mu$ into the space of types over $(a_i)_{i < \mu}$, so there are more than $\mu$ many types over $(a_i)_{i < \mu}$.
\end{proof}
We named real Hilbert spaces as an example of bilinear spaces (\thref{ex:bilinear-spaces}(ii)). Complex Hilbert spaces are not bilinear spaces, because they are required to be conjugate symmetric, that is $[x, y] = \overline{[y, x]}$. Together with linearity in the first argument this means that the form is conjugate linear in the second argument: $[x, \alpha y + \beta z] = \bar{a}[x, y] + \bar{\beta}[x, z]$. This is an example of a \emph{Hermitian space}, which is a generalisation of bilinear spaces. Even more general are \emph{sesquilinear spaces}. It seems likely that similar techniques can be used to study such spaces. We thank Jan Dobrowolski for asking the following question.
\begin{question}
\thlabel{q:sesquilinear-forms}
Can we use positive logic to study sesquilinear or Hermitian spaces over a fixed base field (or even: division ring)? Is the arising theory still simple?
\end{question}

%% file: tex/omega-categoricity.tex
\section{\texorpdfstring{$\omega$}{omega}-Categoricity in positive logic}
\label{sec:omega-categoricity}
In this section we provide a positive version of what is commonly known as the Ryll-Nardzewski theorem in full first-order logic, extending \cite[Theorem 6.5]{haykazyan_spaces_2019}. This will then be used to conclude that, for countable $K$, the theory $T_K^*$ of (symmetric/alternating) $K$-bilinear spaces is $\omega$-categorical, see \thref{cor:bilinear-spaces-omega-categorical}.
\begin{definition}
\thlabel{def:categoricity}
Let $\kappa$ be a cardinal. A theory $T$ is called \emph{$\kappa$-categorical} if it has only one e.c.\ model of cardinality $\kappa$, up to isomorphism.
\end{definition}
The following definition is taken from \cite{haykazyan_spaces_2019}. More precisely, \cite[page 844]{haykazyan_spaces_2019} gives a topological definition of what it means for a set of formulas to be supported. We translate that to a logical property, where we restrict our attention to types (remember, for us these are maximal types, in contrast to \cite{haykazyan_spaces_2019}). This results in \thref{def:supported-type}(ii). We also provide an equivalent property, which is similar to the usual notion of isolated type. In fact, in full first-order logic supported is the same as isolated. In positive logic the latter would not be a good term, as it no longer corresponds to a type being an isolated point in the type space.
\begin{definition}
\thlabel{def:supported-type}
A type $p(\bar{x})$ in finitely many variables is called \emph{supported} if there is a formula $\phi(\bar{x}) \in p(\bar{x})$ such that the following equivalent conditions hold:
\begin{enumerate}[label=(\roman*)]
\item for all $\chi(\bar{x}) \in p(\bar{x})$ we have $T^\ec \models \forall \bar{x}(\phi(\bar{x}) \to \chi(\bar{x}))$,
\item for all $\psi(\bar{x}) \not \in p(\bar{x})$ we have $T \models \neg \exists \bar{x}(\phi(\bar{x}) \wedge \psi(\bar{x}))$.
\end{enumerate}
In this case we call $\phi(\bar{x})$ the \emph{support} of $p(\bar{x})$.
\end{definition}
\begin{lemma}
\thlabel{lem:supported-type}
The conditions in \thref{def:supported-type} are indeed equivalent.
\end{lemma}
\begin{proof}
\underline{(i) $\Rightarrow$ (ii)} Let $\psi(\bar{x}) \not \in p(\bar{x})$ and assume for a contradiction that there is some model $M$ of $T$ with $\bar{a} \in M$ such that $M \models \phi(\bar{a}) \wedge \psi(\bar{a})$. We may assume $M$ to be e.c. Write $q(\bar{x}) = \tp(\bar{a})$, so we have $\phi(\bar{x}) \in q(\bar{x})$ and hence $p(\bar{x}) \subseteq q(\bar{x})$ by our assumption on $\phi$. By maximality of types we then have $p(\bar{x}) = q(\bar{x})$, and hence $\psi(\bar{x}) \in q(\bar{x}) = p(\bar{x})$, which is a contradiction.

\underline{(ii) $\Rightarrow$ (i)} Let $\chi(\bar{x}) \in p(\bar{x})$. If $T^\ec \not \models \forall \bar{x}(\phi(\bar{x}) \to \chi(\bar{x}))$ then by definition of $T^\ec$ there must be an e.c.\ model $M$ with $a \in M$ such that $M \models \phi(\bar{a})$ and $M \not \models \chi(\bar{a})$. So there is a negation $\psi(\bar{x})$ of $\chi(\bar{x})$ such that $M \models \psi(\bar{a})$. As $\chi(\bar{x}) \in p(\bar{x})$ we must have $\psi(\bar{x}) \not \in p(\bar{x})$, so by our assumption on $\phi(\bar{x})$ we have $T \models \neg \exists \bar{x}(\phi(\bar{x}) \wedge \psi(\bar{x}))$. However, this contradicts $M \models \phi(\bar{a}) \wedge \psi(\bar{a})$.
\end{proof}
\begin{definition}
\thlabel{def:atomic-prime-model}
Let $M$ be an e.c.\ model. We call $M$ \emph{atomic} if it only realises supported types. We call $M$ \emph{prime} if every e.c.\ model $N$ is an extension of $M$.
\end{definition}
\begin{fact}[{\cite[Proposition 6.3]{haykazyan_spaces_2019}}]
\thlabel{fact:prime-iff-countable-and-atomic}
Let $M$ be an e.c.\ model of a countable theory $T$ with JEP. Then $M$ is prime if and only if it is countable and atomic.
\end{fact}
\begin{lemma}
\thlabel{lem:supported-n-types-implies-omega-saturated}
If every $n$-type is supported then every e.c.\ model is $\omega$-saturated.
\end{lemma}
\begin{proof}
Let $M$ be an e.c.\ model and let $\Sigma(\bar{x}, \bar{b})$ be finitely satisfiable in $M$, where $\bar{x}$ and $\bar{b} \in M$ are finite. Then there is a realisation $\bar{a}$ in some e.c.\ model $N$ that is an extension of $M$. Set $p(\bar{x}, \bar{y}) = \tp(\bar{a}, \bar{b})$ and let $\phi(\bar{x}, \bar{y})$ be the support of $p(\bar{x}, \bar{y})$. Then $N \models \exists \bar{x} \phi(\bar{x}, \bar{b})$ so because $M$ is e.c.\ we find $\bar{a}' \in M$ with $M \models \phi(\bar{a}', \bar{b})$. As $\phi$ supports $p$ we have that $M \models p(\bar{a'}, \bar{b})$ and hence $M \models \Sigma(\bar{a}', \bar{b})$.
\end{proof}
\begin{remark}
\thlabel{rem:finite-type-spaces}
\thref{thm:omega-categoricity} below provides several equivalent characterisations of being $\omega$-categorical for positive theories. However, compared to the analogous theorem for full first-order logic one important characterisation is missing: namely that the space of $n$-types is finite, for every $n < \omega$. As pointed out in \cite[Example 6.6]{haykazyan_spaces_2019} this is simply no longer equivalent to being $\omega$-categorical in positive logic. In fact, one easily sees that having finite type spaces is equivalent to being $\omega$-categorical and Boolean, where being Boolean follows because the complement of any positively definable set is positively definable using a finite disjunction.

The counterexample from \cite{haykazyan_spaces_2019} is quite simple, so we repeat it here. Consider the theory $T$ with constants $\{c_i\}_{i < \omega}$, asserting that $c_i \neq c_j$ for all $i \neq j$. Then $T$ has a unique e.c.\ model consisting of only interpretations for the constants. We see that $T$ is $\omega$-categorical, but each constant yields a different type, so we have infinitely many $1$-types.
\end{remark}
\begin{theorem}
\thlabel{thm:omega-categoricity}
Let $T$ be a countable theory with JEP. Then the following are equivalent:
\begin{enumerate}[label=(\roman*)]
\item $T$ is $\omega$-categorical,
\item every $n$-type is supported,
\item all e.c.\ models are atomic,
\item all countable e.c.\ models are atomic,
\item every e.c.\ model is $\omega$-saturated,
\item there is a saturated prime e.c.\ model.
\end{enumerate}
\end{theorem}
\begin{proof}
The equivalence between (i), (ii), (iii) and (iv) is \cite[Corollary 6.4 and Theorem 6.5]{haykazyan_spaces_2019}. \thref{lem:supported-n-types-implies-omega-saturated} yields (ii) $\Rightarrow$ (v), while (v) $\Rightarrow$ (i) easily follows from back-and-forth. We are left to prove that (vi) is equivalent to properties (i) to (v).

\underline{(vi) $\Rightarrow$ (iii)} Let $M$ be prime and saturated. By \thref{fact:prime-iff-countable-and-atomic} $M$ is atomic. Let $p$ be a type that is realised in some e.c.\ model $N$. By JEP and saturation $p$ is also realised in $M$ and is thus supported. As $p$ and $N$ were arbitrary we conclude that indeed all e.c.\ models are atomic.

\underline{(i)--(v) $\Rightarrow$ (vi)} Let $M$ be the unique countable model. By (iii) $M$ is atomic and by (v) $M$ is $\omega$-saturated, and hence saturated. Finally, by \thref{fact:prime-iff-countable-and-atomic} $M$ is prime.
\end{proof}
\begin{corollary}
\thlabel{cor:bilinear-spaces-omega-categorical}
Let $K$ be any field. Every e.c.\ model of $T_K^*$ is $\omega$-saturated. If $K$ is at most countable then $T_K^*$ is $\omega$-categorical.
\end{corollary}
\begin{proof}
By \thref{cor:all-types-isolated} every type in finitely many variables is supported and \thref{cor:jep} gives us JEP. So \thref{lem:supported-n-types-implies-omega-saturated} and \thref{thm:omega-categoricity} apply.
\end{proof}

%% file: tex/comparison-to-different-approaches.tex
\section{Comparison to different approaches}
\label{sec:comparison-to-different-approaches}
We consider two other model-theoretic approaches to certain bilinear spaces: Hilbert spaces and the two-sorted approach in full first-order logic. The former is known to be stable, while the latter is known to be non-simple (but is NSOP$_1$). This is in contrast to our main results \thref{thm:category-of-bilinear-spaces-is-simple-aecat} and \thref{thm:positive-theory-simple-unstable}, which claim simplicity and non-stability. In each of these two cases we point out precisely where the difference lies, in terms of the canonical independence relation.
\subsection{Hilbert spaces}
\label{subsec:hilbert-spaces}
From \cite[section 15]{ben-yaacov_model_2008} we know that Hilbert spaces (over the real numbers), studied in the framework of continuous logic, are stable. The canonical independence relation is given by orthogonality. So linear independence cannot be a simple independence relation on the category of Hilbert spaces, as it would have to coincide with orthogonality by canonicity. It is then natural to ask: what property fails? The answer to this question turns out to be \textsc{3-amalgamation}, as we will show in \thref{ex:linear-independence-fails-in-hilbert-spaces}.
\begin{example}
\thlabel{ex:linear-independence-fails-in-hilbert-spaces}
We write $\ind$ for linear independence (as in \thref{def:linear-independence}), and we will show that \textsc{3-amalgamation} fails for $\ind$ in the category of Hilbert spaces. We work in the Hilbert space $\R^3$ with the usual inner product as bilinear form. Let $a = (1, 0, 0)$, $a' = (0,1,0)$, $b = (1, 0, 1)$, $c = (0, 1, 1)$ and $d = (\frac{1}{2}, \frac{1}{2}, 2)$. Set $A = \linspan{ad}$, $A' = \linspan{a'd}$, $B = \linspan{bd}$, $C = \linspan{cd}$ and $D = \linspan{d}$. So we have a commuting diagram as below (ignoring the dashed arrows), where every arrow is an inclusion, except for $f$, which is defined by $f(d) = d$ and $f(a) = a'$, and then extend linearly.
\[
% https://tikzcd.yichuanshen.de/#N4Igdg9gJgpgziAXAbVABwnAlgFyxMJZARgBoAGAXVJADcBDAGwFcYkQAdDgJQD0BmEAF9S6TLnyEUAFgrU6TVuwBqw0SAzY8BIuVLF5DFm0QgAgmrFbJRfvsOKTnHgMsbx2qSVIAmB8fYAYTdNCR0ZX38lUy4+QRErMK89fiinABEQjxsUO1SaI2iQACFheRgoAHN4IlAAMwAnCABbJD0QHAgkMgUA0wALEBpGegAjGEYABWzwkEYYOpwhuawwJyh6OH6Kt0aWpB8aTqQ7ObGJ6etZ+cXdptbEQ46uxHbCpzrlkfGpmak5hZLBIgPYPU7HRA9d7sSpfc6-K7-G5LYardabbZQO77RCyZ5tAqOJBgZiMRjDeGXJLsZHYh54iEAVkJfRJZIpPypnhpgLpSGZ+MhLKKWDhnL+PNuqLW7A2Wx2wNBSAAbEcXgB2YUmNnks7ixGSoHqJWIVWCvHQxA6jkXCWmWmK+4qtUHLXE0m6762g323mOnGawUCy3WvXe6m+27+h6BiGnEMem0IiMAqOUIRAA
\begin{tikzcd}[row sep=tiny]
                                    & \R^3 \arrow[rrr, "h", dashed] &  &                              & V                            \\
A \arrow[rrr] \arrow[ru, "f"]       &                               &  & \R^3 \arrow[ru, "g", dashed] &                              \\
                                    & C \arrow[uu] \arrow[rrr]      &  &                              & \R^3 \arrow[uu, "i", dashed] \\
D \arrow[rrr] \arrow[ru] \arrow[uu] &                               &  & B \arrow[ru] \arrow[uu]      &                             
\end{tikzcd}
\]
Furthermore, we have $A \ind_D^{\R^3} B$, $B \ind_D^{\R^3} C$ and $A' \ind_D^{\R^3} C$. Noting that $A' = f(A)$ \textsc{3-amalgamation} would give us the dashed arrows such that everything commutes. We will view $i$ as a genuine inclusion and write $a^* = g(a) = h(a')$.

Set $v = 2(b+c - d)$, so in $\R^3$ this is just $(1, 1, 0)$. We calculate:
\begin{align*}
[a^* - v, a^* - v] &= [a^*, a^* - v] - [v, a^* - v] \\
&= ([a^*, a^*] - [a^*, v]) - ([v, a^*] - [v, v]) \\
&= (1 - ([a^*, 2b] + [a^*, 2c] - [a^*, 2d])) - ([v, a^*] - [v, v]) \\
&= (1 - (2 + 2 - 1)) - ([2b, a^*] + [2c, a^*] - [2d, a^*] - 2) \\
&= -2 - (2 + 2 - 1 -2) \\
&= -3,
\end{align*}
here we have used the definition of $a^*$ and commutativity of the above diagram multiple times for simplifications like $[a^*, 2b] = [g(a), i(2b)] = [g(a), g(2b)] = [a, 2b] = 2$.

So we have found an element of $V$, namely $a^* - v$, such that $[a^* - v, a^* - v] = -3$. This means that the form on $V$ is not positive definite. So $V$ cannot be a Hilbert space and we conclude that \textsc{3-amalgamation} fails for linear independence $\ind$ in the category of Hilbert spaces.
\end{example}
Of course, we could apply \thref{thm:3-amalgamation} to the diagram in \thref{ex:linear-independence-fails-in-hilbert-spaces} to find $V$ together with the dashed arrows. We will just get a bilinear form that is not positive definite. In fact, we can give an explicit description of $V$ and $a^*$ (which completely determines the diagram). We take $V = \R^4$, and let the bilinear form be defined by $[(x,y,z,w), (x',y',z',w')] = xx' + yy' + zz' - ww'$. Now we can take $a^* = (1\frac{1}{2}, 1\frac{1}{2}, -\frac{1}{2}, \frac{\sqrt{15}}{2})$. It is then straightforward to verify that this does indeed form a solution to the \textsc{3-amalgamation} problem.
\subsection{Failure of simplicity when the field varies}
\label{subsec:failure-of-simplicity-when-the-field-varies}
Bilinear spaces over some infinite field $K$ have been studied in the framework of full first-order logic by considering a two-sorted theory $T_\infty^K$: one sort for the vector space and one for the field, which has to be elementarily equivalent to $K$ in the language of rings. This comes with two disadvantages: the theory is going to be at least as complicated as the theory of the field and the field varies between different models of the theory.

Even when we restrict ourselves to algebraically closed fields, for which we write $T_\infty^\ACF$, the resulting theory will be non-simple, as was established in \cite[Proposition 7.4.1]{granger_stability_1999}. Later, in \cite[Corollary 6.4]{chernikov_model-theoretic_2016}, it was shown that $T_\infty^\ACF$ is NSOP$_1$. In terms of independence relations this means that the canonical independence relation has all properties that a simple independence relation has, except for \textsc{Base-Monotonicity}. We give an example of how precisely \textsc{Base Monotonicity} fails.
\begin{definition}
\thlabel{def:first-order-theory-of-bilinear-spaces}
Let $T_\infty^\ACF$ be the full first-order theory of an infinite dimensional non-degenerate bilinear space, either symmetric or alternating, over an algebraically closed field of characteristic other than $2$. We have two sorts, $V$ and $K$, for the vector space and the field respectively. The sort $V$ has the language of abelian groups on it, and $K$ has the language of rings on it. Furthermore, we have a function symbol $K \times V \to V$ for scalar multiplication and a function symbol $[\cdot, \cdot]: V \times V \to K$ for the bilinear form. 
\end{definition}
We use $T_\infty^\ACF$ to refer both to the symmetric and alternating version, as it makes no difference in what follows. We introduce some notation. For any set $A$ in some model $M$ we write $K(A)$ for the restriction of $A$ to the field sort, so $K(A) = A \cap K(M)$. Similarly $V(A) = A \cap V(M)$. As the field can now vary with the models we need to include it in the notation of the linear span as well: for a field $K_0 \subseteq K(M)$ we write $\linspan{A}_{K_0}$ for the $K_0$-linear span of $V(A)$.

The canonical independence relation in $T_\infty^\ACF$ was first described in \cite[Proposition 9.37]{kaplan_kim-independence_2020}, where the base was restricted to models. In \cite[Corollary 8.13]{dobrowolski_sets_2020} some corrections were made and the independence relation was extended to arbitrary sets, resulting in the following fact. We write $\dcl(X)$ for the model-theoretic definable closure of $X$ and $\ind^\ACF$ is algebraic independence.
\begin{fact}
The canonical independence relation $\ind^K$ (the $K$ here stands for \emph{Kim-independence}) for $T_\infty^\ACF$ is as follows: for $A, B, C \subseteq M$ we have $A \ind_C^{K, M} B$ if and only if $\linspan{AC}_{K(M)} \cap \linspan{BC}_{K(M)} = \linspan{C}_{K(M)}$ and $K(\dcl(AC)) \ind_{K(\dcl(C))}^{\ACF, M} K(\dcl(BC))$.
\end{fact}
\begin{example}
\thlabel{ex:failure-of-base-monotonicity-in-first-order}
We show that \textsc{Base Monotonicity} fails for $\ind^K$ in $T_\infty^\ACF$. The relation $\ind^K$ consists of two parts: linear independence in the vector space sort and algebraic independence in the field sort. The failure will take place in the algebraic independence, and comes from taking the definable closure of $AC$.

Fix some model $M = (V_0, K_0)$. Let $M \preceq N$ with $v,w \in V(N)$ be such that:
\begin{enumerate}[label=(\roman*)]
\item $v$ and $w$ are $K(N)$-linearly independent over $V_0$,
\item $[v, w] = b$ is transcendental over $K_0$,
\item $[v, v] = [w, w] = [v, x] = [x, v] = [w, x] = [x, w] = 0$ for all $x \in V_0$.
\end{enumerate}
Let $K_b$ be the smallest algebraically closed field containing $K_0 b$. Set $A = (\linspan{V_0 v}_{K_0}, K_0)$, $B = (\linspan{V_0 w}_{K_b}, K_b)$ and $C = (\linspan{V_0 w}_{K_0}, K_0)$. Each of $A, B, C$ is algebraically closed. We quickly see that $A \ind_M^{K, N} B$ and $M \subseteq C \subseteq B$. We also have that $b = [v, w] \in \dcl(AC)$ and so $K(\dcl(AC)) \nind_{K(\dcl(C))}^{\ACF, N} K(\dcl(BC))$, because $b$ is transcendental over $K_0 = K(\dcl(C))$ and by construction $b \in K(\dcl(BC))$. We thus conclude that $A \nind_C^{K,N} B$, so \textsc{Base Monotonicity} fails.
\end{example}